\documentclass[11pt,oneside]{amsart}
\usepackage{mdframed}
\usepackage{amsfonts,amscd,amssymb,amsmath,amsthm}
\usepackage{etex}
\usepackage{todonotes}
\usepackage{mathrsfs}
\usepackage{caption,graphicx,stmaryrd}
\usepackage[textheight= 615pt, textwidth=360pt]{geometry}
\usepackage{marginnote}
\usepackage{verbatim}
\usepackage{mathtools}
\usepackage{enumerate}
\usepackage{esint}

\usepackage{wasysym}
\usepackage{nccmath}

\usepackage[english]{babel}
\usepackage[utf8]{inputenc}

\usepackage{xparse}

\usepackage{wrapfig,subcaption}
\usepackage[colorlinks,citecolor=blue]{hyperref}

\usepackage[
    backend=bibtex,
    style=alphabetic,
    sortlocale=auto,
    natbib=true,
    url=false, 
    doi=true,
    eprint=false,
    safeinputenc=true,
    giveninits=true,
    isbn=false,
    maxbibnames=10
]{biblatex}
\definecolor{Maroon}{rgb}{1.0,0.2,0.5}

\def\polhk#1{\setbox0=\hbox{#1}{\ooalign{\hidewidth
    \lower1.5ex\hbox{`}\hidewidth\crcr\unhbox0}}} 

\makeatletter
\def\paragraph{\@startsection{paragraph}{4}%
	\z@\z@{-\fontdimen2\font}%
	{\normalfont\bfseries}}
\makeatother

\newtheorem{theorem}{Theorem}[section]
\newtheorem{lemma}[theorem]{Lemma}

\newtheorem{proposition}[theorem]{Proposition}
\theoremstyle{definition}
\newtheorem{definition}[theorem]{Definition}

\newtheorem{question}[theorem]{Question}
\newtheorem{claim}[theorem]{Claim}
\newtheorem{remark}[theorem]{Remark}

\renewcommand{\Re}{\operatorname{Re}}

\newcommand{\sinc}{\operatorname{sinc}}

\newcommand{\R}{\mathbb{R}}
\newcommand{\C}{\mathbb{C}}
\newcommand{\D}{\mathbb{D}}
\newcommand{\A}{\mathcal{A}}
\newcommand{\cE}{\mathcal{E}}

\newcommand{\N}{\mathbb{N}} 

\newcommand{\Z}{\mathbb{Z}}
\newcommand{\T}{\mathbb{T}}

\newcommand{\frks}{\mathfrak{s}}
\newcommand{\regs}{s}
\newcommand{\LP}{\operatorname{L}}

\newcommand{\Matrices}[1]{ \mathbb{M}\left( {#1} \right) } 
\newcommand{\SAMatrices}[1]{ \mathbb{M}_{SA}\left( {#1} \right) } 
\DeclareDocumentCommand{\Matrices}{ O{r} O{n} }{ \mathbb{M}_{#2} \left( {#1} \right) }
\DeclareDocumentCommand{\SAMatrices}{ O{r} O{n} }{ \mathbb{M}_{#2,SA} \left( {#1} \right) }

\newcommand{\mart}{\mathcal{M}}
\newcommand{\homp}{\mathcal{H}}

\newcommand{\Exp}{\mathbb{E}}

\newcommand{\prob}{\mathbb{P}}
\newcommand{\filt}{\mathscr{F}}
\newcommand{\Gfilt}{\mathscr{G}}

\renewcommand{\Pr}{\prob}
\DeclareDocumentCommand \one { o }
{%
\IfNoValueTF {#1}
{\mathbf{1}  }
{\mathbf{1}\left\{ {#1} \right\} }%
}

\newcommand{\Var}{\operatorname{Var}}

\newcommand{\Unif}{\operatorname{Unif}}
\newcommand{\As}{\ensuremath{\operatorname{a.s.}}}
\newcommand{\lawequals}{\overset{\mathscr{L}}{=}}
\DeclareDocumentCommand{\Prto} {o} {
\IfNoValueTF {#1}
 {\overset{\Pr}{\longrightarrow}}
 { \xrightarrow[ #1 \to \infty]{\Pr }}
}
\DeclareDocumentCommand{\Asto} {o} {
\IfNoValueTF {#1}
 {\overset{\operatorname{a.s.}}{\longrightarrow}}
 {
 \xrightarrow[ #1 \to \infty]{\operatorname{a.s.} }
 }
}
\DeclareDocumentCommand{\Mgfto} {o} {
\IfNoValueTF {#1}
{\overset{\operatorname{mgf}}{\longrightarrow}}
{ \xrightarrow[ #1 \to \infty]{\operatorname{mgf} }}
}

\DeclareDocumentCommand{\Wkto} {o} {
\IfNoValueTF {#1}
 {\overset{(d)}{\longrightarrow}}
 { \xrightarrow[ #1 \to \infty]{(d) }}
}

\DeclareDocumentCommand{\LPto} { D<>{1} o} {
\IfNoValueTF {#2}
 {\overset{\operatorname{\LP^{#1}}}{\longrightarrow}
 }
 { \xrightarrow[ #2 \to \infty]{\LP^{#1} }}
}

\newcommand{\dd}{\mathrm{d}}

\newcommand{\intpart}[1]{{\left\lfloor#1\right\rfloor}}

\begin{document}

\title[Image of random analytic functions]{The image of random analytic functions:
coverage of the complex plane via branching processes}
\author{Alon Nishry}
\address{School of Mathematical Sciences, Tel Aviv University}
\email{alonish@tauex.tau.ac.il}
\author{Elliot Paquette}
\address{Department of Mathematics, McGill University}
\email{elliot.paquette@mcgill.ca}
\thanks{
}
\date{\today}

\begin{abstract}
  We consider the range of random analytic functions with finite radius of convergence.
  We show that any \emph{unbounded} random Taylor series with rotationally invariant coefficients has dense image in the plane.
  We moreover show that if in addition the coefficients are complex Gaussian with sufficiently regular variances, then the image is the whole complex plane.
  We do this by exploiting an approximate connection between the coverage problem and spatial branching processes.
  This answers a long-standing open question of J.-P. Kahane, with sufficient regularity.
\end{abstract}

\maketitle

\section{Introduction}

In this work, we consider random Taylor series 
\begin{equation}
\label{eq:powerseries}
F(z) = \sum_{n \geq 1} \zeta_n a_n z^n,
\end{equation}
where $(\zeta_n, n \geq 1)$ is a sequence of independent, identically--distributed, rotationally invariant (i.e., $
\zeta_n \sim e^{i \theta} \zeta_n $ for any $\theta \in [0, 2\pi]$) complex random variables with $\Exp |\zeta_n|^2=1$ and $(a_n : n \geq 1)$ is a non-random nonnegative real-valued sequence (for convenience we set $a_0 = 0$).  We let $\rho_F > 0$ be the radius of convergence of $F$, which is almost surely given by $\liminf a_n^{-1/n}.$

The basic problem of value-distribution theory of analytic functions is to study the set of solutions $\{z_n\}$ of the equation $F = a$ for $a \in \C$.
An important special case is the problem of determining the image of $F$, in particular when $F$ is unbounded, when is $F(\D)=\C?$

Typically this has been approached using Jensen's formula, which allows one to prove that $F$ takes every $b \in \C$ infinitely often by proving the divergence of the logarithmic integral
\[
  \int_0^1 \log| F(re(\theta)) - b| \dd \theta
  \quad
  \text{where}
  \quad
  e(\theta) = e^{2\pi i \theta }
\]
as $r$ tends $\rho_F$.  The difficulty lies in proving that this holds simultaneously over an uncountable set, but this has been proven under the necessary condition that $\Exp |F(r)|^2 = \sum a_n^2 r^{2n} \to \infty$ as $r\to \rho_F$ (see \cite{nazarov2014log} and \cite{Kahane} for the complex Gaussian case).

Hence we can specialize to the case that: (a) the radius of convergence is finite, and, without loss of generality, equal to $1$, and (b)  $\sum a_n^2 < \infty.$  Under these assumptions, $F$ is an analytic function on the unit disk $\D$, in the Hardy class $H^2$ almost surely, which in particular means that its boundary values $F(e(\theta)) \coloneqq \lim_{r \to 1} F(re(\theta))$ exist for almost every $\theta \in [0,1]$.  The stochastic processes that can result on the boundary include many canonical stochastic processes, both continuous and discontinuous. This was the setting for much early theory on sample path properties of stochastic processes (see e.g. \cite[Sec. 16.3]{Kahane}).

A clear obstruction to the image $F(\D) = \C$ is that $F$ is bounded, which by a theorem of \cite{Billard}, is equivalent to the almost sure continuity of $F$ on the unit circle $\T$ (and hence also on $\overline{\D}$).
There is a wide range of unbounded random analytic functions with $\sum a_n^2 < \infty$,
and a priori this image could be any unbounded open connected set.
Our first theorem shows that the range of an unbounded $F$ is always dense.
\begin{theorem}
  For any $F$ with $\sum a_n^2 < \infty$, $\limsup a_n^{1/n} = 1$ if $F$ is unbounded on $\D$ almost surely then $F(\D)$ is dense.
  \label{thm:dense}
\end{theorem}
\begin{remark}
When $\sum a_n^2 = \infty$ it is known that almost surely $F$ has \emph{no} radial limit at a.e. boundary point $e(\theta)$ (e.g. \cite[Sec. 5.5]{Kahane}), and the result follows by a short argument (see \cite[p. 127]{Kahane68}). If $F$ is not a.s. dense, then, with positive probability, there is some $b \in \C$ such that $(F(z) - b)^{-1}$ is bounded and therefore has radial limits along a.e. radius, which leads to a contradiction.
\end{remark}

It was proven in great generality that when $\sum a_n^2 =\infty$, then $F(\D) = \C$. A natural leap is to ask if this is in fact true for any unbounded random analytic function.
In fact, Kahane states \cite[Section 10]{kahane1992some} (see also \cite[Question 1.4]{nazarov2014log}):
\begin{quote}
  There is also a remaining question, even in the case of Gaussian Taylor series.  Is it true that either such a series represents a.s.\ a continuous function on the closed disk $|z| \leq 1$, or that it maps a.s.\ onto the whole plane $\C$?
\end{quote}

We shall show that this insinuation of Kahane holds for Gaussian analytic functions (GAFs) given by a Taylor series with regular coefficient sequences. A GAF is the special case of \eqref{eq:powerseries} in which $\zeta_n$ are standard complex Gaussians (i.e.\ with density $\tfrac{1}{\pi}e^{-|z|^2}$ with respect to Lebesgue measure on $\C$).


We consider square-summable regular sequences of nonnegative coefficients $\{ a_n \}_{n\ge1}$ of the form
\[
a_n = n^{-\alpha} \Omega(\log n), \qquad \alpha \ge \tfrac12,
\]
where $\Omega$ is a regularly-varying function; explicitly,
\begin{equation}\label{eq:an-coeffs-def}
a_n = n^{-\alpha} (\log(n+1))^{-\beta} \omega(\log (n+1)), \qquad \beta \in \R,
\end{equation}
where $\omega$ is a nonnegative, slowly-varying function at infinity, meaning
\[
\omega: (0,\infty) \to (0,\infty), \qquad \lim_{x\to\infty} \frac{\omega(a x)}{\omega(x)} = 1, \quad \text{for every } a > 0.
\]

Denote
\begin{equation}\label{eq:sj-def}
\regs_j^2 = \sum_{n=2^j}^{2^{j+1} - 1} a_n^2 \asymp 2^{j(1-2\alpha)} j^{-2\beta} \omega^2(j).
\end{equation}
The assumption $\sum_n a_n^2 = \sum_j \regs_j^2 < \infty$, necessitates $\alpha > \tfrac12$ or $\alpha = \tfrac12$ and $\beta \ge \frac12$.
For coefficients $(a_n)_n$ given by \eqref{eq:an-coeffs-def}, the proof of Theorem 1 in \cite[Chap. $7$]{Kahane} implies that $F$ is \emph{bounded} when
\[
\sum_{k=1}^\infty 2^{k/2} \bigg( \sum_{j=2^k}^{2^{k+1}-1} \regs_j^2 \bigg)^{1/2} < \infty,
\]
hence we require $\alpha = \frac12$, and $\beta \le 1$. In fact, together with Theorem 1 in \cite[Chap. $8$]{Kahane}, a 
necessary and sufficient condition for the GAF with coefficients \eqref{eq:an-coeffs-def} to be \emph{unbounded} on $\D$, is $\sum_j \regs_j = \infty$.

We then define:
\begin{definition}\label{def:regular}
  We say that a GAF is critical-regular
  if $\left\{ a_n \right\}$
  is given by
  \[
    a_n = n^{-1/2}(\log (n+1))^{-\beta} \omega( \log (n+1)),
    \quad n \geq 1,
  \]
  where $\beta \in [\tfrac12,1]$ and $\omega$ is a nonnegative slowly varying function.
\end{definition}

We prove the following statement, which thus answers the question of \cite{kahane1992some} under sufficient regularity.
\begin{theorem}\label{thm:coverage}
  Any unbounded, critical-regular GAF $F$ satisfies $F(\D) = \C$ almost surely.
\end{theorem}



\begin{remark}
In Theorem~\ref{thm:coverage} we work under the assumptions
\begin{equation*}
  \beta \in [\tfrac12, 1],
  \quad 
  \sum_{m=1}^\infty m^{-2\beta}  \omega^2(m) < \infty
  \quad\text{and}\quad
  \sum_{m=1}^\infty m^{-\beta}  \omega(m) = \infty.
\end{equation*}
We emphasize that when $\beta < 1$ there is in fact \emph{no} restriction on the function $\omega$. If $\omega$ is chosen such that the left series above diverges, then the result of Theorem \ref{thm:coverage} follows (in a stronger form) from \cite[Theorem 1.3]{nazarov2014log}.
\end{remark}


\subsection{Proof strategy}

We give a proof that relies primarily on statistical properties of the random analytic function, rather than analytical ones, and exploits an inherent branching process structure built into $F$.
To describe this, we begin by grouping the terms of $F$ into blocks of exponentially growing order.
We let $L \in \N$ and define for any $k \in \N$
\begin{equation}\label{eq:Xblock}
  X_k(\theta,r) 
  =X_{k,L}(\theta,r) \coloneqq \sum_{n=2^{L(k-1)}}^{2^{Lk}-1}
  \zeta_n a_n r^n e(n \theta)
  \quad
  \text{and}\quad
  \frks_k^2 \coloneqq \sum_{n=2^{L(k-1)}}^{2^{Lk}-1} a_n^2.
\end{equation}
We will also let $X_k(\theta) = X_k(\theta,1)$.
We refer to these random polynomials as the blocks of $F$, and we use the notation for all $0 \leq m < n \leq \infty$
\begin{equation}\label{eq:Fblock}
  F_{m,n} \coloneqq \sum_{k=m+1}^n X_k.
\end{equation}
We shall still write $F_{0,\infty}(\theta, r)$ for $F(re(\theta)).$ \\

\paragraph{Branching process heuristic.}
As families over $k$, these polynomials $\{X_k\}$ are independent.  More critically, as one varies $\theta \in \R/\Z$, these polynomials $X_k(\theta,1)$ roughly behave as though the arcs $[j2^{-L(k-1)},(j+1)2^{-L(k-1)}]$ are independent over $j$ and moreover locally constant.  This remains true for $X_k(\theta,r)$ provided $r \geq 1-2^{-Lk}.$  If one adopts this heuristic, the process $F_{0,n}(\theta)$ becomes an encoding of a \emph{branching random walk}.

This branching random walk $\mathcal{Z}_{n,j}$ is a doubly indexed collection of processes, where for each fixed `point in time' $n$ the process $(\mathcal{Z}_{n,j} : n)$ is a complex random walk with rotationally invariant increments of variance $\frks_n^2$.  In the index $j$, the collection grows as $2^{L(n-1)}$, and each particle splits into $2^L$ new, independent and identically distributed random walks at each step.

The assumption $\sum_k \frks_k^2 < \infty$ means that any \emph{fixed} random walk remains bounded, while the unboundedness assumption of the GAF leads \emph{exceptional} random walkers grow unbounded. To show the density of this image, we exploit that if the process is unbounded, we can find, over a sufficiently sparse sequence $n_k$, a specially chosen walker $j_k$ so that $(\mathcal{Z}_{n_k,j_k} : k)$ is a random walk with rotationally invariant increments that satisfy 
\begin{equation}\label{eq:densewalker}
  \sum_k \Exp |\mathcal{Z}_{n_{k+1},j_{k+1}}-\mathcal{Z}_{n_k,j_k}|^2 =\infty
  \quad
  \text{and}
  \quad
  \Exp |\mathcal{Z}_{n_{k+1},j_{k+1}}-\mathcal{Z}_{n_k,j_k}| \Asto[k] 0.
\end{equation}
Such a random walk has dense image in the plane (for a more precise formulation, see Lemma \ref{lem:bmcoupling}). Translating this result back to the GAF means there is a \emph{non-tangential} path $\Gamma$ to a (random) boundary point on the circle $\T$, so that the image $F(\Gamma)$ is dense in $\C$.\\

\paragraph{Coverage heuristic.}
As for covering the plane, using analytic function theory, we can show that the image of $F$ can roughly be considered as a fattened version of the branching process where around each walker $\mathcal{Z}_{n,j}$ we put a disk of radius about $n^{-1-2\beta}$ (c.f.\ Proposition~\ref{prop:cover}). 
 On a technical level, we also need to show sufficient continuity in the GAF to be able to compare to the branching heuristic.  See for example Lemma \ref{lem:windowbound}.

One might hope that by fattening the single dense walker $\{ \mathcal{Z}_{n_k,j_k} : k\}$ referenced in 
\eqref{eq:densewalker} by disks of radius $n_k^{-1-2\beta},$ this image covers the plane.  However this union of disks always has finite volume.

Furthermore, the question of whether or not there exists a single dense walker that covers the plane has the following analogous analytical question:
\begin{question}\label{q:sector}
  For which critical-regular GAFs does there exist a Stolz sector 
  \[
	  T(\theta, M) \coloneqq \{e^{i\theta}z : z \in \D, |1-z| \leq M(1-|z|)\}
  \]
  for some $M >1$ such that $F(T(\theta,M)) = \C$?
\end{question}
\noindent In particular, it is not clear that this should have the same threshold as the coverage question of when $F(\D) = \C$.

So, returning to the heuristic for coverage, we need to exploit the existence of the full branching process available.  Now it suffices to show that in fact, for any $u \in \C,$ $F(\D)$ contains $\D(u,1)$ with probability $1$.
To handle all unbounded branching processes, we need to find many random walkers that approach $u \in \C$ and we need to effectively estimate the rate at which we approach $u$.

In the case $\beta < 1$ this problem is relatively easy.  Since we have $s_n \asymp n^{-\beta}\omega(n)$, the variance contained in the random walk $(\mathcal{Z}_{m,j}-\mathcal{Z}_{n,j} : m \geq n)$ is on the order of $n^{-2\beta + 1}$, up to slowly varying factors.  We will have on the order of $2^n$ many $j$ so that $|\mathcal{Z}_{n,j}| \leq C$, as this is typical behavior for a walker with bounded variance; call these typical $j$, $\mathcal{G}_n$.
For any of these $j \in  \mathcal{G}_n$, we can then ask how many travel to a small neighborhood of $u$ (say, closer than $n^{-1000}$) by time $n^2$.  The probability that this occurs for $1$ walker is $e^{-\tilde{\omega}(n) n^{2\beta-1}}$ for some slowly varying function $\tilde{\omega}$.  

However since $\beta < 1$ and we have exponentially many walkers in $\mathcal{G}_n$, we have many (in fact exponentially many in $n$) walkers close to $u$ by time $n^2$.  As we fatten these by $\bigl(n^{2}\bigr)^{-2\beta-1}$ disks, $u$ is contained in the image of the fattened random walk with probability that can be sent to $1$ by sending $n$ to infinity. As stated, this is enough to contain the point $u$.  The coverage argument requires covering $\D(u,1)$ with probability going to $1$.  But in this case, while the ball around $u$ that we cover has shrinking radius, the probability with which this occurs is $1-e^{-cn}$, locally uniformly in $u$ and so by a union bound we can cover $\D(u,1)$.

This argument can be turned into a coverage argument for the GAF, but we do not show this formally (the proof we give in  Section \ref{sec:coverage} is instead based on tools which 
 are also useful for the case $\beta=1$).
It furthermore exposes $\beta=1$ as qualitatively different in that the variance contained in the tail  $(\mathcal{Z}_{m,j}-\mathcal{Z}_{n,j} : m \geq n)$ is not always enough to overwhelm lack of progress in approaching the point $u$ by time $n$.

So in the case that $\beta=1$, we break the walk into three time intervals, the ``homing'', ``spawning'' and ``coverage'' intervals.  Using the notation of Section \ref{sec:coverage}, we refer to these time intervals as $[0,k_h)$, $[k_h,k_s)$ and $[k_s,k_s^2)$ respectively.  These are related by $k_s = k_h + O(\log k_h),$ and $k_h$ will be a paramter which we then take to $\infty.$

  In the homing interval, we greedily thin the branching random walk, by selecting walkers that tend to go towards $u$.  In fact, by adjusting the targets of this swarm of random walkers, we select a random collection of walkers which form a $(k_h)^{-4}$ net of $\D(u,1)$ with probability tending to $1$ as $k_h \to \infty$.  The basic ``homing'' strategy is implemented in Section~\ref{sec:homing} for the GAF, while the construction of the net is done in Proposition \ref{prop:homingnet}.

  In the spawning interval, we then let these random walkers, which were successful in homing towards a net of $\D(u,1)$, spawn more walkers.  As the amount of time ellapsed between $k_s$ and $k_h$ is short, one typically would expect these walkers to only move by about $\eta \coloneqq k_h^{-1}\omega(k_h)\sqrt{\log k_h}$.  This is much larger than $k_h^{-4},$ and so after spawning, we can only guarantee that we have a $\eta$--net of $\D(u,1)$.  However, we have, for each point $x \in \D(u,1)$ almost $2^{k_s-k_h}=k_h^{C}$ many walkers that are within distance $\eta$ of this $x$.  Moreover this constant $C>0$ can be as large as we like by adjusting the precise definition of $k_s$.

  Finally, in the coverage interval, we then use these walkers to cover all of $\D(u,1).$  Each increment $(\mathcal{Z}_{k_s^2,j}-\mathcal{Z}_{k_s,j})$ has standard deviation on order $k_h^{-1/2}{\omega}(k_h) \geq \eta k_h^{1/4}$, and we have $k_h^{C}$ many walkers at distance $\eta$ from each point in $\D(u,1)$.  Hence for $C$ sufficiently large, and any desired point $x \in \D(u,1)$ we can find a walker within distance $k_s^{-\Lambda}$ (for any desired fixed $\Lambda>0$) of $x$ with probability $1-k_s^{-C'}$, where $C'$ can be made as large as desired by increasing $C.$
  Thus, by a union bound (conditionally on the outcome of the homing interval), we cover the whole disk $\D(u,1).$

  \begin{remark}
    The heuristic above applies equally to the GAF coverage problem as to a coverage problem just about spatial branching processes.   Our method shows that some class of fattened branching process images covers the plane.  
    Below we have formulated a general question about fattened branching processes, which could be viewed as the probabilistic essence of the random analytic function coverage question.  
  \end{remark}

  \begin{question}
    Suppose $\{\mathcal{Z}_{n,j}\}$ is a spatial binary branching process, where $n$ denotes time and $1 \leq j \leq 2^{n-1}$.  Suppose the increments of the walks are rotationally invariant and have nonincreasing variances $\mathfrak{s}_n^2$.  Let $\mathfrak{r}_n$ be a nonincreasing sequence, and let $\mathcal{F}$ be the union of balls around $\mathcal{Z}_{n,j}$ of radius $\mathfrak{r}_n$ over all $n$ and $j$.  Under what conditions is $\mathcal{F}=\R^2?$
  \end{question}
%
%

\subsection{Discussion}

We have shown that unbounded random analytic functions with rotationally invariant coefficients have dense image, and further, if they are unbounded regular GAFs then they have image covering the whole plane.
It is a natural question if these theorems extend to the case of Rademacher analytic functions (whose coefficients have independent signs).
\begin{question}
  For an unbounded Rademacher analytic function $F :\D \to \C$ with $\sum a_n^2 <\infty$, is $F(\D)=\C$?
\end{question}

Gaussian and rotationally invariant coefficients are simpler to handle. Under the assumptions of finite radius of convergence and $\sum a_n^2 = \infty$, the coverage question for general symmetric distributions, in particular the Rademacher case, was settled in \cite{nazarov2014log}. Earlier, Offord \cite{offord1972distribution} gave a positive answer for certain non-lattice distributions of the random variables $\{\xi_n\}$ (e.g. standard complex normal, see also \cite[Chap. 13]{Kahane}). For discrete distributions there were also previous results, under additional assumptions on the sequence $\left\{ a_n \right\}$ (e.g.\ \cite{jakob1983distribution, murai1981value}).

We have formulated Theorem \ref{thm:coverage} for a class of regularly varying coefficient sequences.  This in particular includes arbitrarily slowly diverging GAFs (i.e.\ those for which $\sup_\theta |F_{0,m}(e(\theta))|$ diverges arbitrarily slowly with $m$.  There is a component of our strategy which is fundamentally quantitative, in that we need to show the homing process works sufficiently quickly that we can still cover. Thus in many places in the argument (in fact in all parts), we rely on computations which are made possible by the asymptotic form of $\frks_k \asymp_L k^{-\beta} \omega(k)$.  We also rely on the finer structure of the variance within blocks in one part (in the proof of Lemma \ref{lem:hittingu}); we note that the branching heuristic we have given \emph{requires} that within blocks $[2^{L(k-1)},2^{Lk}]$ the coefficient sequence $\{a_n\}$ is well distributed.

The question of characterizing the image of analytic functions in the unit disk has been raised in the context of non-random Taylor series, specifically (Hadamard) lacunary power series
\[
f(z) = \sum_{k=1}^\infty a_k z^{n_k}, \qquad n_k \in \N, \: a_k \in \C \quad \text{and} \quad n_{k+1} / n_k \ge q > 1,
\]
where we assume the radius of convergence is equal to $1$. 
It has generally been true that lacunary power series and random power series have parallel theories, with the lacunary theory generally established far earlier and often more simply. For example, lacunary series are unbounded if and only if $\sum_k |a_k| = \infty$ (e.g. \cite[Theorem 3.6.6]{Grafakos143rdEd}), compared with the complicated condition found by Dudley, Fernique, and Marcus and Pisier (e.g. \cite[Chap. 15]{Kahane}).

Paley \cite{paley1933lacunary} raised the same question as Kahane's in the context of lacunary series. For such functions, Murai \cite{murai1981value} proved that unboundedness implies the image is the entire complex plane (this was proved earlier by G. Weiss and M. Weiss \cite{Weiss62} for $q$ sufficiently large).

Lacunary functions $f$ have additionally been the subject of other statistical studies, such as that of \cite{Jensen2014} (with $a_k = 1, n_k = 2^k$), which shows in randomly selecting a point uniformly $\omega \in \T$ uniformly at random, the path $\tau \mapsto f( (1-2^{-n\tau})\omega)$ has a scaling limit to a complex Brownian motion on sending $n \to \infty$.  Thus in a sense, the typical boundary point of a lacunary series, when approached radially, have dense image in the plane.  In contrast, unbounded GAFs with $\sum a_n^2 < \infty$ have boundary values which are a.e. finite.  We nonetheless show in the proof of Theorem \ref{thm:dense} that there exist exceptional boundary points which when approached from the interior (nontangentially, albeit not radially), the GAF has dense image. We leave as an open question (cf. \cite[Chap. 13]{Kahane}, when $\sum_n a_n^2 = \infty$):
\begin{question}
  For an unbounded GAF $F :\D \to \C$ with $\sum a_n^2 <\infty$, does there almost surely exist $\omega \in \T$ so that the image $\{ F(r\omega) \colon r \in [0,1] \}$ is dense in $\C$?
\end{question}


\subsection{Organization}

In Section \ref{sec:density} we give a self-contained proof of the density Theorem \ref{thm:dense}.
In Section \ref{sec:prelim} we develop some preliminary lemmas that we use in the coverage argument. 
In Section \ref{sec:homing} we develop the homing mechanism.
In Section \ref{sec:branching} we show how to cover a desired ball using the approximate branching structure of the GAF.
In Section \ref{sec:coverage} we complete the proof of Theorem \ref{thm:coverage}.
Finally in Appendix~\ref{app:rw} we include a general lemma about the density of random walks on $\C$ of shrinking rotationally invariant steps.

\subsection{Acknowledgment}
We thank Fedor Nazarov for hinting at the relation between GAFs and branching processes.
The authors thank the US-Israel Binational science foundation, grant number 2018341, which supported this work.  
The first author acknowledges support from ISF grant 1903/18.
The second author acknowledges support from NSERC Discovery grant RGPIN-2020-04974.

\section{Density}\label{sec:density} 


In this section we prove Theorem \ref{thm:dense}.  Hence throughout, we suppose that $F\in H^2(\D)$ is a random Taylor series with radius of convergence $1$, which is unbounded almost surely. We begin with some simple observations about the random series. 

\begin{lemma}
  Suppose that $F$ is unbounded almost surely.  Then for any $n,N \in \N$ and any $\theta \in \R,$
  \[
    \sup_{|\psi - \theta| \leq \tfrac{1}{2N}} | F_{n,m}(\psi) | \Asto[m] \infty.
  \]
  \label{lem:smallarc}
\end{lemma}

\begin{remark}
  The block structure of $F_{n,m}$ that was introduced in \eqref{eq:Fblock} is not important here. In this lemma, one could just as well use $F_{n,m}$ as the Gaussian polynomial $\sum_{j=n+1}^m \zeta_j a_j z^j$, and moreover, throughout this section one could use this alternative definition.   
\end{remark}

To prove Lemma \ref{lem:smallarc}, we will rely on the following estimate:
\begin{claim}\label{clm:Fejer_prop}
Let $F = \sum_n a_n z^n$ be an analytic function in $H^2(\D)$, and let $P_m$ be the corresponding algebraic Fej\'{e}r polynomial $P_m = F \star K_m = \sum_n (1-n/m)_+ z^n$ of degree $m$, where $K_m$ is the analytic Fej\'{e}r kernel. Then, if $I \subset [-\pi,\pi]$ is any interval, and $I/2$ is the interval with same center as $I$, and length $|I|/2$, then for all $m$ sufficiently large (depending only on $|I|$)
\[
\sup_{t \in I/2} |P_m(e^{it})| \le \sup_{t\in I} |F(e^{it})| + \frac{C \|F\|_2}{m|I|},
\]
where $C > 0$ is a numerical constant.
\end{claim}
\begin{proof}
We may assume that $M = \sup_I |F| < \infty$ and that $I$ is centered at $0$. Then, for $t \in I/2$
\[
|P_m(e^{it})| = \frac{1}{2\pi} \left| \int_{-\pi}^{\pi} K_m(s) F(e^{i(t-s)}) \dd s \right|
\le \frac{M}{2\pi} \int_{I/4} K_m + \sup_{(I/4)^c} K_m \cdot \| F \|_2,
\]
since the Fej\'{e}r kernel is nonnegative. It remains to notice that
$\frac{1}{2\pi} \int_{I/4} K_m \le 1$ and $\sup_{(I/4)^c} K_m \lesssim \frac{1}{m|I|}$.
\end{proof}

\begin{proof}[Proof of Lemma \ref{lem:smallarc}]
  As it is a polynomial $F_{0,n}$ is almost surely bounded on $\D.$  Hence by assumption and the maximum principle 
  \[
  \sup_{\psi \in [0,1]} |F_{n,\infty}(\psi)| = \infty.
  \]

  Notice it is enough to prove the result for rational $\theta$ and any $n, N$. Moreover, by rotation invariance it is sufficient to prove that
  \[
  	\sup_{|\psi| \leq 1/(2N)} |F_{n,m}(\psi)| \Asto[m] \infty.
  \]
  Assume this is false, then using the fact it is a tail event, we see that for some $C$ sufficiently large, the event
  \[
  	\liminf_{m \to \infty} \sup_{|\psi| \leq 1/(2N)} |F_{n,m}(\psi)| \le C.
  \]  
  has probability at least $1 - 1/(8N)$. We may also assume that on the same event $\|F\|_2 \le C$ (and in particular $\|F_{n,m}\|_2 \le C$ uniformly in $n,m$).
  
  Now define the Fej\'{e}r polynomials $G_{n,m}$ with respect to $F_{n,\infty}$, and notice they are also the Fej\'{e}r polynomials w.r.t. sufficiently long truncation $F_{n,m^\prime}$ of $F_{n,\infty}$. Therefore, by Claim \ref{clm:Fejer_prop}, for all $m$ sufficiently large (depending on $N, C$) we have
  \[
  \sup_{|\psi| \leq 1/(4N)} |G_{n,m}(\psi)| \le 2 \sup_{|\psi| \leq 1/(2N)} |F_{n,m^\prime}(\psi)|, \qquad \forall m^\prime \ge m.
  \]
  Hence, taking $m^\prime\to\infty$ over an appropriate subsequence, we find that with probability at least $1 - 1/(6N)$,
  \[
    	\sup_m \sup_{|\psi| \leq 1/(4N)} |G_{n,m}(\psi)| \le 2 C.
  \]  
  By rotation invariance and the union bound we conclude that with probability at least $1/2$ it holds
  \[
    	\sup_m \sup_{\psi \in [0,1]} |G_{n,m}(\psi)| \le 2 C.
  \]
  But on this event $F_{n,\infty}$ is bounded by $2C$ which leads to a contradiction.



\end{proof}

We also observe that the minimum modulus of $F_{n,m}$ can be made small on any interval, simply by virtue of the coefficients being square summable.
\begin{lemma}
  For any $n \in \N$ and any $\theta \in \R$ and all $t > 0,$
  \[
    \liminf_{m \to \infty} \Pr[ |F_{n,m}(\theta)| > t ]
    \leq
    \Pr[ |F_{n,\infty}(\theta)| > t].
  \]
  Hence for any $t > 0,$
  \[
    \lim_{n \to \infty}
    \liminf_{m \to \infty} \Pr[ |F_{n,m}(\theta)| > t ]
    =0.
  \]
  \label{lem:min}
\end{lemma}
\begin{remark}
By the contraction principle (see e.g. \cite[Sec. 4.2]{LedouxTalagrand}) it holds that $\Pr[ |F_{n,m}(\theta)| > t ] \leq 2 \Pr[ |F_{n,\infty}(\theta)| > t]$ for every $m \ge n$.
\end{remark}
\begin{proof}
  The first statement is the weak convergence $F_{n,m}(\theta) \Wkto[m] F_{n,\infty}(\theta),$ and the second is the weak convergence $F_{n,\infty}(\theta) \Wkto[n] 0.$
\end{proof}

Finally we observe that $F$ enjoys many distributional symmetries.
\begin{lemma}
  For any $m,n \in \N$ with $m > n$ (or with $m =\infty$) we have the identity in law
  \[
    (F_{n,m}(\theta,r) : re(\theta) \in \overline{\D}) \lawequals (\Omega F_{n,m}(\theta,r) : re(\theta) \in \overline{\D}),
  \]
  where $\Omega$ is a uniformly distributed element of $\T$ independent of $F_{n,m}.$  Hence for any random variables $(\Theta,R)$ with $R \cdot e(\Theta) \in \overline{\D}$ and measurable with respect to $\sigma( \{|F_{n,m}(\theta,r)| : re(\theta) \in \overline{\D}\}),$ 
  \[
    (\arg(F_{n,m}(\Theta,R)) , |F_{n,m}(\Theta,R)|)
  \]
  are independent and $\arg(F_{n,m}(\Theta,R))$ is uniformly distributed on $[0,2\pi].$
  \label{lem:rotation}
\end{lemma}
\begin{proof}
  The first observation is an immediate corollary of the rotation invariance and independence of $\left\{ \zeta_n \right\}.$  The second follows as $F_{n,m}(\Theta,R) \lawequals \Omega F_{n,m}(\Theta,R)$ as a random function, and $|F_{n,m}(\Theta,R)|$ is independent of $\Omega.$  
\end{proof}

Similarly, we note that for any $\psi \in \R,$  $(\theta,r) \mapsto F_{n,m}(\theta + \psi,r)$ has the same law as when $\psi = 0.$ 

Below, we denote by $\|\cdot\|_\infty$ the supremum norm in $\D$.

\begin{lemma}
  Let $\epsilon > 0$ be any number and $\{q_k\}$ be any sequence with $(k^{\epsilon} q_k) \downarrow 0$ and $\sum_{k=1}^\infty q_k^2 = \infty.$  
  If $F\in H^2(\D)$ is unbounded almost surely,
  then there are deterministic sequences $n_k \uparrow \infty$ (with $n_0 =0$) and $N_k \uparrow \infty$ (with $N_0 = 1$) and a random sequence $\theta_k \in [0,1]$ with $\theta_k \in \sigma \big( \left\{ |F_{0,n_k}| \right\} \big)$, and ${|\theta_k - \theta_{k-1}|} \leq N_k^{-1}2^{-k}$, so that the following hold:
  \begin{enumerate}[\textnormal{(}i\textnormal{)}]
    \item\label{ubi-new} The derivative is not large
    \[
    \Pr\big[\max_{1 \leq j < k} \|F'_{n_{j-1},n_{j}}\|_\infty > N_k\big] \leq 2^{-k}.
    \]
    \item\label{ubii-new} The increments are not large in that
      \[
      \Pr[|F_{n_k,\infty}(0)| > q_{k+1}] \leq 2^{-k-2},
      \text{ and }\,
      \Pr[|F_{n_{k-1},n_k}(\theta_{k-1})| > q_k] \leq 2^{-k}.
      \]
    \item\label{ubiii-new} The increments are not too small in that
      \[
      \Pr\Big[\sup_{|\psi - \theta_{k-1}| \leq {N_k}^{-1}2^{-k}} | F_{n_{k-1},n_k}(\psi) | \leq q_k\Big] \leq 2^{-k}.
      \]
  \end{enumerate}
  \label{lem:ub-new}
\end{lemma}
\begin{proof}
  We construct the sequences $\{n_k\}$, $\{N_k\}$ and $\left\{ \theta_k \right\}$ inductively.  Suppose that $n_{k-1}$ and $\theta_{k-1}$ have been chosen.  For $k=1,$ we take $n_0=\theta_0=0.$ 

  For any $1 \leq j < k,$ $F'_{n_{j-1},n_{j}}$ is an almost surely continuous random function on $\overline{\D}$.  Hence, we can find an integer $N_k$ sufficiently large that property \eqref{ubi-new} is satisfied.
  Then from Lemma \ref{lem:smallarc},
  \[
    \sup_{|\psi - \theta_{k-1}| \leq {N_k}^{-1}2^{-k}} | F_{n_{k-1},m}(\psi) | \Asto[m] \infty.
  \]
  Now we pick $n_k$ sufficiently large that property \eqref{ubiii-new} and the first part of property \eqref{ubii-new} are satisfied. Note that these hold for \emph{all} $n_k$ sufficiently large. Now the remaining part of property \eqref{ubii-new},
  \[
  	\Pr[|F_{n_{k-1},n_k}(\theta_{k-1})| > q_k] \leq 2^{-k},
  \]
  is possible because we already have
  \[
  	\Pr[|F_{n_{k-1},\infty}(0)| > q_k] \leq 2^{-k-1},
  \]
  and by Lemma \ref{lem:min}, and the independence of $F_{n_{k-1},\infty}$ from $\theta_{k-1}$,
  \[
    \begin{aligned}
      \liminf_{n_k \to \infty}
      \Pr[|F_{n_{k-1},n_k}(\theta_{k-1})| > q_k]
      &\leq
      \Pr[|F_{n_{k-1},\infty}(\theta_{k-1})| > q_k] \\
      &=
      \Pr[|F_{n_{k-1},\infty}(0)| > q_k]
      \leq 2^{-k-1}.
  \end{aligned}
  \]
  Hence there is a choice of $n_k$ large enough that all three properties occur.
  
  Within the interval $\left\{ \psi \in \R/\Z : |\psi - \theta_{k-1}| \leq{N_k}^{-1}2^{-k}\right\},$ we let $\theta_k$ be the closest point to $\theta_{k-1}$ so that $|F_{n_{k-1},n_k}(e(\theta_k))|=q_k,$ breaking ties in a measurable way. If no such $\theta_k$ exists, we let $\theta_k = \theta_{k-1}.$  
\end{proof}

Let $\left\{ n_k \right\}$ and $\left\{ \theta_k \right\}$ be the sequence from Lemma \ref{lem:ub-new}.  In terms of these sequences, let $r_k = |F_{n_{k-1},n_k}(e(\theta_k))|$ and let $\omega_k = \frac{F_{n_{k-1},n_k}(e(\theta_k))}{|F_{n_{k-1},n_k}(e(\theta_k))|}.$  Then from Lemma \ref{lem:rotation}, and Borel-Cantelli the sequence $\left\{ (r_k,\omega_k) \right\}$ are independent random variables with $k^{\epsilon} r_k \to 0$ and $\sum r_k^2 = \infty$ almost surely. We then introduce the process
\begin{equation}
  S_k \coloneqq \sum_{j=1}^k r_j \omega_j \quad k = 0,1,2,\dots,
  \label{eq:Sk}
\end{equation}
which is dense in $\C$ almost surely.  We give a proof for convenience in Lemma \ref{lem:bmcoupling} (cf. \cite[Chap. 12, Sec. 6]{Kahane}).

We now show the density of the image.
\begin{proof}[Proof of Theorem \ref{thm:dense}]
  
  We let $\lambda_\ell = 1-N_\ell^{-1}2^{-\ell}$ which increases to $1$.

  Observe that by construction of $(S_k :k)$ for any $k < j$
  \begin{equation*}
    S_j-S_k 
    =
    \sum_{\ell=k+1}^j (S_\ell-S_{\ell-1})
    =
    \sum_{\ell=k+1}^j F_{n_{\ell-1},n_\ell}(\theta_\ell).
  \end{equation*}
  Using Lemma \ref{lem:ub-new} part \eqref{ubi-new} and Borel--Cantelli, there is a random $k_0$ so that for any $k+1 \leq \ell < j$ with $k \geq k_0,$
  \[
    |
    F_{n_{\ell-1},n_\ell}(\theta_\ell, \lambda_j)
    -
    F_{n_{\ell-1},n_\ell}(\theta_j,\lambda_j )
    |
    \leq 
    \sum_{p=\ell+1}^{j-1}
    \|F_{n_{\ell-1},n_\ell}'\|_\infty |\theta_p - \theta_{p-1}|
    \leq 2^{-\ell}.
  \]
  We also have that
  \[
     |
    F_{n_{\ell-1},n_\ell}(\theta_\ell)
    -
    F_{n_{\ell-1},n_\ell}(\theta_\ell,\lambda_j)
    |
    \leq 
    \|F_{n_{\ell-1},n_\ell}'\|_\infty
    (1 - \lambda_j)
    \leq 2^{-\ell}.
  \]
  Hence we have for any $k \geq k_0.$ 
  \begin{equation}\label{eq:BMtube}
    |S_j-S_k 
    -F_{n_k,n_j}(\theta_j, \lambda_j)|
    \leq \sum_{\ell=k+1}^j 2^{-\ell+1} \leq 2^{-k+1}.
  \end{equation}

  Note that by its definition in Lemma \ref{lem:ub-new}, the sequence $\{ \theta_\ell \}$ is Cauchy in $\R/\Z$. So there is a $\theta_*$ so that $F_{0,n_k}(\theta_\ell,\lambda_\ell) \Asto[\ell] F_{0,n_k}(\theta_*,1).$  Then define for any $z \in \C$ and $k,p \in \N$ the stopping time (with respect to the discrete filtration generated by $\{S_j : j \in \N\}$)
  \begin{equation}\label{eq:sigmap}
    \sigma_p = \inf\left\{ \ell > p : |S_\ell - S_k - (z - F_{0,n_k}(\theta_*))| \leq 2^{-p} \right\},
  \end{equation}
  then each $\sigma_p < \infty$ almost surely from neighborhood recurrence. Also, by the first part of Lemma~\ref{lem:ub-new}~\eqref{ubii-new}, the contraction principle, and Borel--Cantelli, we have that $|F_{n_{\sigma_p},\infty}(\theta_{\sigma_p},\lambda_{\sigma_p})| \Asto[p] 0.$

  Finally we put everything together (c.f.\ \eqref{eq:BMtube}, \eqref{eq:sigmap}).
  \[
    \begin{aligned}
    &\liminf_{p \to \infty}
    |F(\theta_{\sigma_p},\lambda_{\sigma_p})-z|\\
    &\leq 
    \liminf_{p \to \infty}
    \{|F_{0,n_k}(\theta_{\sigma_p},\lambda_{\sigma_p}) + F_{n_k, n_{\sigma_p}}(\theta_{\sigma_p},\lambda_{\sigma_p})-z|+ |F_{n_{\sigma_p},\infty}(\theta_{\sigma_p},\lambda_{\sigma_p})|\} \\
    &\leq
    \liminf_{p \to \infty}
    \{
      |S_{\sigma_p}\!-\!S_k - (z - F_{0,n_k}(\theta_{\sigma_p},\lambda_{\sigma_p}))|
      +
      |S_{\sigma_p}\!-\! S_k -F_{n_k,n_{\sigma_p}}(\theta_{\sigma_p},\lambda_{\sigma_p})|
    \} \\
    &\leq
    \liminf_{p \to \infty}
    \{
      |S_{\sigma_p}\!-\! S_{k} - \bigl(z - F_{0,n_k}(\theta_{*})\bigr)|
      +
      |F_{0,n_k}(\theta_*) - F_{0,n_k}(\theta_{\sigma_p},\lambda_{\sigma_p})|
      +
      2^{-k}
    \} \\
    &\leq 2^{-k}.
    \end{aligned}
  \]
  As we may do this for any $k$ sufficiently large, we conclude that there is a sequence of points (given by $\omega_k =
  \lambda_{\sigma_p^{(k)}}e(\theta_{\sigma_p^{(k)}})$ so that $F(\omega_k) \Asto z.$
  As we may do this for all rational complex numbers, this completes the proof.
\end{proof}

\begin{remark}
  The proof above uses $\lambda_\ell = 1 - N_\ell^{-1}2^{-\ell}$ which therefore constructs a path that approaches the limit $\theta_* = \lim \theta_\ell$ nontangentially, since $|\theta_\ell - \theta_*| \leq N_\ell^{-1}2^{-\ell+1}.$
\end{remark}

\section{Coverage Preliminaries}\label{sec:prelim}

\subsection*{Concentration}
We shall use concentration of measure for Gaussian processes, and it will be convenient to formulate some of these statements using the sub-gaussian Orlicz norm $\|\cdot \|_{\psi_2}$ which is defined by for a random variable $X$ by 
\[
  \|X\|_{\psi_2} = \inf\{ t \geq 0 : \Exp e^{|X|^2/t^2} \leq 2\}.
\]
For a general reference on its properties, see for example \cite[Section 2.5]{Vershynin}. 
This leads directly to a tail bound, by Markov's inequality
\[
  \Pr(|X| \geq t) \leq 2\exp(-t^2/\|X\|_{\psi_2}^2),
\]
and so we also get that $\Exp |X| \leq C \|X\|_{\psi_2}$ for some absolute constant $C>0$.
Moreover, finiteness of the Orlicz norm leads to a bound for the moment generating function: for all $\lambda > 0$ and for some absolute constant $C>0$
\[
\Exp \exp(\lambda X) 
\leq \exp\Bigl( \lambda \Exp X + C\lambda^2 \|X\|_{\psi_2}^2\Bigr)
\leq \exp\Bigl( C(\lambda \|X\|_{\psi_2} + \lambda^2 \|X\|_{\psi_2}^2)\Bigr).
\]

\subsection*{Gaussian homing}

The next lemma shows that if we take two weakly correlated complex Gaussians, and we select the Gaussian that is closer to a desired target direction, we create a sub-gaussian random variable with a mean which points towards the desired target.  This is the basis of how we construct the homing process.
\begin{lemma}\label{lem:homing}
  Let $(Z_1,Z_2)$ be standardized jointly complex normals with correlation 
  $\rho \coloneqq \Exp( Z_1 \overline{Z_2}).$
  There are positive constants $\rho_{H}, \delta, r_0$ so that 
  for all $u \in \C$ with $|u| \geq r_0/2$ and 
  all $|\rho| < \rho_{H}$,
  then with $Z_*$ given by whichever of $Z_j$ maximizes $\Re( \overline{u} Z_j)$ for $j\in \{1,2\}$,
  \[
    \Exp
    \exp\bigl(
    \lambda |Z_* - u|
    \bigr)
    \leq
    \exp\bigl(
    \lambda(|u| - \delta)
    +C_H \lambda^2
    \bigr)
    \quad
    \text{for all } \lambda \geq 0
  \]
  for an absolute constant $C_H.$
\end{lemma}
\begin{proof}
  We may assume $u \geq r_0/2,$
  by the rotational invariance of the law of $(Z_1,Z_2)$.
  Let $\cE = \cE(\rho) = \Exp |Z_*-u|$ and let $Y_* = |Z_* - u| - \cE$.  
  The claim follows from showing that $\cE \leq u-\delta$, for some $\delta>0$, all $u \geq r_0/2$ and all $|\rho| < \rho_H$, and from showing that $\|Y_*\|_{\psi_2} \leq C_H$ for an absolute constant $C_H$.
  Indeed, having shown this, we have
  \[
    \Exp e^{\lambda Y_*} \leq e^{C'_H \lambda^2}
  \]
  for some other constant $C'_H$ and all $\lambda \in \R$.  Rearranging this equation, and restricting to $\lambda \geq 0,$ the claim follows for $|Z_*-u|.$

  We start with bounding the mean.
  There exists a probability space supporting $(Z_1,Z_2)$ and a pair of \emph{independent} standard complex normals $(W_1,W_2)$ so that
  \[
    \lim_{|\rho| \to 0} 
    \Exp
    \exp
    \bigl(
    |Z_1 - W_1| + |Z_2-W_2|
    \bigr)
    =1.
  \]
  Hence if we let $W_*$ be the $W_j$ with the larger real value,
    \[
    \lim_{|\rho| \to 0} 
    \Exp
    \exp
    \bigl(
    |Z_* - W_*|
    \bigr)
    =1.
  \]
  Thus by uniform integrability
  it suffices to show for some $\delta \in (0,1)$ so that
  \[
    \cE(0)
    =
    \Exp
    |W_* - u|
    \leq
    u  - 2\delta
  \]
  for all $u \geq r_0$.
  
  We can express $W_* = X+iY$ where $X \coloneqq \max\{X_1,X_2\}$ for three independent real normals $\{X_1,X_2,Y\}$ of variance $\tfrac12$.  Then
  \[
    \Exp|W_* - u|^2
    =
    \Exp[{(u  - X)^2 + Y^2}]
    =
    u^2 - \sqrt{\tfrac{2}{\pi}} u + 1.
  \]
  Thus, by Jensen's inequality, for $u \ge 7$,
  \[
    \cE(0)
    =
    \Exp
    |W_* - u|
    \leq
    u-\tfrac14.
  \]

It remains to show that the sub-gaussian norm of the centered random variable $Y_*$ is uniformly bounded.  Observe that if $V_*$ is an independent copy of $Z_*$
\[
  \begin{aligned}
  \| Y_*\|_{\psi_2}
  =
  \| |Z_* - u| - \mathcal{E}(\rho)\|_{\psi_2}
  &=\| \Exp(|Z_* - u| - |V_*-u| \mid Z_*)\|_{\psi_2}\\
  &\leq
  \| \Exp(|Z_* - V_*|~\mid Z_*) \|_{\psi_2},
  \end{aligned}
\]
where we have used the reverse triangle inequality.  
Then since $\||Z_*|\|_{\psi_2} \leq \|Z_1\|_{\psi_2} +\|Z_2\|_{\psi_2}$ and 
$\Exp(|Z_* - V_*|~\mid Z_*) \leq |Z_*| + 2\Exp |V_*|$.
So we conclude there is an absolute constant $C_H$ so that
\(
  \| Y_*\|_{\psi_2}
  \leq C_H.
\)
\end{proof}

\subsection{GAF tools}

We shall also use some tools developed specifically for Gaussian analytic functions.
The first gives upper bounds for general GAFs on the disk.

\begin{lemma}[{\cite[Lemma 2]{BNPS18}}]\label{lem:GAFtail} 
  Let $G$ be a GAF on $\mathbb{D}$, and $s \in(0, \delta)$. Put
  $$
  \sigma_G^2(r)=\sup _{r \mathbb{D}} \mathbb{E}\left[|G|^2\right] .
  $$
  Then, for every $\lambda>0$,
  $$
  \mathbb{P}\left\{\sup _{r \mathbb{D}}|G|>\lambda \sigma_G(r+s)\right\} \leq C s^{-1} \exp \left(-c \lambda^2\right)
  $$
  and in particular
  $$
  \mathbb{P}\left\{\sup _{r \mathbb{D}}|G|>\lambda \sigma_G\left(\frac{1}{2}(1+r)\right)\right\} \leq C(1-r)^{-1} \exp \left(-c \lambda^2\right).
  $$
\end{lemma}

The second tool is rather for analytic functions, but it is the basic lemma that allows to cover a disk.

\begin{lemma}[{\cite[Lemma 1]{Fuchs}}]\label{lem:cover}
If $g(z)$ is analytic in $\mathbb{D}(0, R)$, and if
$$
\left|g'(z)\right| \leq S, \quad \forall|z|<R, \quad\left|g'(0)\right| \geq A>0,
$$
then $g(z)$ assumes in $\mathbb{D}(0, R)$ every value $w$ lying in the disk
$$
\mathbb{D}\left(g(0), CA^2 R S^{-1}\right),
$$
where $C>0$.
\end{lemma}

\section{Homing}
\label{sec:homing}

In this section we implement the homing mechanism, which will produce short arcs on which a GAF is uniformly close to a desired value.  For how we apply it, this will be applied to the polynomial (recall \eqref{eq:Fblock})
\[
  F_{m,n}(\theta,r)
  \coloneqq
  \sum_{k=m+1}^n X_k(\theta,r),
\]
and we shall refer to the \emph{homing angle process over the blocks} $m$ to $n$ as one defined for this GAF.

We now define for any $\theta \in \R$ 
\[
  \rho_k(\theta) \coloneqq \Exp \big[ X_k(0) \overline{X_k(\theta)} \big] \frks_k^{-2}.
\]
We let $\mathfrak{r}_k$ be the radius
\[
  \mathfrak{r}_k \coloneqq 1-2^{-Lk},
\]
and we assume throughout that $L \geq 2$.
We would like to treat the polynomial $X_k(\psi,r)$ as almost constant in a window 
\begin{equation}\label{eq:window}
  \mathcal{W}_k(\theta,\alpha) \coloneqq\{ (\psi,r) : r \in [\mathfrak{r}_k,1],|e(\theta)-e(\psi)| \leq \alpha 2^{-L k}\}.
\end{equation}
To control the errors from this approximation by a constant, we introduce the local H\"older-$\tfrac12$ constant over the window $\mathcal{W}_k(\theta,\alpha)$ for any $j < k$:
\begin{equation}\label{eq:Delta}
  \begin{aligned}
    \Delta_{j,k}(\theta,\alpha)
    \coloneqq
  &\sup_{\mathcal{S}(\alpha)}
  \left\{
  \frac{
  \bigl|
  F_{j,k}(\theta+\tfrac{\psi_1}{2^{Lk}}, 1-\tfrac{s_1}{2^{Lk}})
  -F_{j,k}(\theta+\tfrac{\psi_2}{2^{Lk}}, 1-\tfrac{s_2}{2^{Lk}})
  \bigr|}
  {
    \sqrt{|\psi_1-\psi_2|
    +|s_1-s_2|}
  }
  \right\}
  , \\
  &
  \text{where }
  \mathcal{S}(\alpha)
  \coloneqq \{(\psi_1,s_1,\psi_2,s_2) : |\psi_1|,|\psi_2| \leq \alpha, 
  s_1,s_2 \in [0,1]\}
  \end{aligned}
\end{equation}
We also write as shorthand $\Delta_{k}(\theta,\alpha) \coloneqq \Delta_{k-1,k}(\theta,\alpha).$

Recall $\rho_H$ from Lemma \ref{lem:homing}. We shall show in Lemma \ref{lem:homingdecoupling}, that there is a sequence of angles $\mathfrak{t}_{m+1},\mathfrak{t}_{m+2},\dots$ such that
\(
  \rho_k(\mathfrak{t}_k) < \rho_H
\)
for all $k \geq m+1$ and such that $|\mathfrak{t}_{k}| \leq 2^{L(k-1)}$.
Define the homing angle process $(\A_k, \homp_k)$ from initial angle $\A_m$ by the rule for $k \geq m$
\begin{equation}\label{eq:homing}
  \begin{aligned}
  \A_{k+1} 
  &=
  \A_{k}
  +
  \chi_k
  \operatorname{argmax}\bigg\{ \Re( (\overline{\homp_k-u}) X_{k+1}( \A_k + x)) : x \in \{0, \mathfrak{t}_{k+1}\} \bigg\}\\
  \homp_{k+1} 
  &=
  F_{m,k+1}(\A_{k+1}),\quad\text{and}\\
  \chi_k &= \one\{|F_{m,k}(\A_{k})-F_{m,k}(\A_{k}+\mathfrak{t}_{k+1})|\leq \sqrt{\eta}\frks_{k+1}\},
\end{aligned}
\end{equation}
where 
$\homp_{m}=0$ and $\eta > 0$ is a constant to be determined.
By rotation invariance the law of the process $(\A_k - \A_m)_k$ does not depend on the initial angle $\A_m$.
\begin{remark}
  We note that the location of the angle $\A_j$ for $m \leq j \leq n$ differs from the initial angle $\A_m$ by at most $\sum_{k=m+1}^j 2^{-L(k-1)} \leq 2^{-Lm + 1}$ and therefore $\mathcal{W}_k(\A_k,1) \subseteq \mathcal{W}_m(\A_m,4)$.
\end{remark}

Our main homing result is the following.  We show that the homing process succeeds (in that it approaches $u$) and moreover that one can simultaneously guarantee that the H\"older constant is good in a neighborhood of this point: 
\begin{proposition}\label{prop:homing}
  For any $\gamma \in (0,\tfrac14)$
  and for any $\alpha \geq 1$
  there is a $c>0$ so that for all $n\geq m/a$ (with $a$ the constant from Lemma \ref{lem:Gammamonotonicity} below) such that
$L_{m,n} \geq \max\{2|u|,\sqrt{V_{m,n}}\}$
\[
    \Pr\biggl(  
      |F_{m,n}(\A_n) - u| + \Delta_{m,n}(\A_n,\alpha)
    > 2\frks_n^{1-\gamma}
    \biggr) \!\leq\! \exp(-c \min\{ L_{m,n}^2/V_{m,n}, n^{\gamma/4}\}).
  \]
\end{proposition}
\begin{remark}
  Note that this event depends only on the block of the GAF from $m+1$ to $n$, and so we may take $u$ depending on $\filt_m$ and we have the same bound for the conditional probability $\Pr( \cdot \mid \filt_m)$.
\end{remark}

\subsection{Control for the H\"older constant}

We shall show that by choosing $L$ large enough, we can guarantee that the H\"older constant is small while the correlations are small.
In the first step we detail how we use $L$ to improve the properties of the GAF blocks. 
\begin{lemma}\label{lem:homingdecoupling}

  For all $\eta > 0$
  there is an integer $L_0$ so that for all $L \in \N$ with $L \geq L_0$
  \begin{enumerate}
    \item
      if $\Theta_k \lawequals \Unif([-2^{-L(k-1)},2^{-L(k-1)}])$ then for all $k \in \N$
      \[
	\Exp |\rho_k(\Theta_k)| < \rho_H, 
      \]
      where $\rho_H$ is the same as in Lemma~\ref{lem:homing}.
    \item 
      For any $\alpha \geq 1$ there is a constant $C_\alpha$ so that for any $\theta$
      \[
	\|
	\Delta_k(\theta,\alpha)
	\|_{\psi_2}
	\leq C_\alpha \eta \frks_k.
      \]
  \end{enumerate}
\end{lemma}
\begin{proof}\,\\

  \paragraph{Correlation coefficient.}
  
  The correlation coefficient can be written as
  \[
    \rho_k(\Theta)
    \frks_k^2
    =
    \sum_{n=2^{L(k-1)}}^{2^{Lk}-1}
    a^2_n e(-n \Theta).
  \]
  Taking expectation of the second moment with respect to $\Theta$,
  \begin{equation}\label{eq:peyton}
    \Exp
    | \rho_k(\Theta) |^2
    \frks_k^4
    =
    \sum_{n_1,n_2}
    a^2_{n_1}
    a^2_{n_2}
    \sinc
    ( (n_1-n_2) 2^{-L(k-1)} ),
  \end{equation}
  where we recall that
  \[
    \sinc(x) 
    \coloneqq \frac{1}{2}\int_{-1}^{1} e(xt) \dd t
    = \frac{\sin(2\pi x)}{2\pi x}.
  \]
  Also recalling \eqref{eq:sj-def} and \eqref{eq:Xblock}, we have
  \[
  \frks_k^2 = \sum_{j=L(k-1)}^{Lk-1} \regs_j^2 \asymp L \regs_{Lk-1}^2, \,\,\text{ as }\,\, \regs_j \asymp j^{-\beta} \omega(j).
  \]
  Hence we bound \eqref{eq:peyton} by
  \[
    \Exp
    | \rho_k(\Theta) |^2
    \frks_k^4
    \lesssim
    \sum_{j=L(k-1)}^{Lk-1}
    \regs_j^2(\regs_j^2 + \regs_{j+1}^2)
    +\sum_{j=L(k-1)}^{Lk-1}
    \regs_j^2
    \sum_{\ell = j+2}^{Lk-1}
    2^{-\ell+L(k-1)}\regs_\ell^2.
  \]
  Using that $(\regs_j)_j$ is regularly varying,
  \begin{equation}\label{eq:rhok}
    \Exp
    | \rho_k(\Theta) |^2
    \frks_k^4
    \lesssim
    \frks_k^2
    \regs_{L(k-1)}^2.
  \end{equation}
 \paragraph{Oscillation bound.}
  Fix $\theta \in \R$ and define
  \[
    D\big((\psi_1,t_1),(\psi_2,t_2)\big) 
    \!\!\coloneqq\!\!
    \frac{X_{k}(\theta+\tfrac{\psi_1}{2^{Lk}}, 1-\tfrac{t_1}{2^{Lk}})
      -X_{k}(\theta+\tfrac{\psi_2}{2^{Lk}}, 1-\tfrac{t_2}{2^{Lk}})}
    {
      \sqrt{|\psi_1-\psi_2|
      +|t_1-t_2|}
    }.
  \]
  Then, for $|\psi_1|,|\psi_2| \leq \alpha$ and $t_1,t_2 \in [0,1]$,
  \[
    \begin{aligned}
    \Var(D)
    &\leq
    \sum_{n=2^{L(k-1)}}^{2^{Lk}-1}
    \frac{2a^2_n
      \Bigl(
      \sin^2\left(\tfrac{\pi n (\psi_1-\psi_2)}{ 2^{Lk}}\right) 
      +
      \Bigl((1-\tfrac{t_1}{2^{Lk}})^{n}
      -(1-\tfrac{t_2}{2^{Lk}})^{n}\Bigr)^2
      \Bigr)
    }
    {|\psi_1-\psi_2| + |t_1-t_2|} \\
    &\lesssim
    |\psi_1-\psi_2|
    \sum_{j={L(k-1)}}^{Lk-1}
    \regs_j^2
    2^{-2(Lk - j)}
    +|t_1-t_2|
    \sum_{j={L(k-1)}}^{Lk-1}
    \regs_j^2
    2^{-2(Lk - j)}.
  \end{aligned}
  \]
  By increasing $L$, we can therefore bound this by
  \[
    \sum_{j={L(k-1)}}^{Lk-1}
    \regs_j^2
    2^{-2(Lk - j)}
    \leq
    c \eta^2 \frks_k^2
  \]
  for some absolute constant $c >0$ by making $L_0$ sufficiently large, which then leads to 
  \[
    \Var\Big(D\big((\psi_1,t_1),(\psi_2,t_2)\big)\Big)
    \leq C(|\psi_1-\psi_2| + |t_1-t_2|)\eta^2 \frks_k^2.
  \]
  As this bound holds on a compact of $(\psi_i,t_i)$ of diameter depending only on $\alpha$, by standard Gaussian tail bounds, the supremum is controlled (see for example \cite[Theorem 8.5.5]{Vershynin}),
  i.e.\ there is a constant $C_\alpha > 0$ such that
  \[
    	\|
	\Delta_k(\theta,\alpha)
	\|_{\psi_2}
	\leq 
	C_\alpha
	\eta \frks_k.
  \]
\end{proof}

\subsection{Proof of the homing}

We now define a sequence of stopping times with respect to $\filt_k = \sigma( X_j : j \leq k)$
which are given by
\[
  T_j
  \coloneqq
  \inf\{ k \geq j : |u-\homp_{k}| < r_0 \frks_{k+1} \}
  \quad\text{for any $j \geq m$},
\]
where $r_0$ is as in Lemma \ref{lem:homing}.

Define for $k \geq m$ and any $\lambda \geq 0$ 
\[
  \begin{aligned}
  &\mart_k
  \coloneqq
  \exp\bigl(
  \lambda \bigl(
  |u-\homp_k|
  +
  L_{m,k}-D_{m,k}\bigr)
  -\lambda^2 V_{m,k}
  \bigr)
  \quad
  \text{where} 
  \quad\\
  &
  L_{m,k} =\frac{\delta}{2} \sum_{j=m+1}^k \frks_k,
  \quad
  D_{m,k} = C_D\sum_{j=m+1}^k (1-\chi_{k-1})\frks_k,
  \quad
  \text{and}\quad\\
  &V_{m,k} = C_H\sum_{j=m+1}^k \frks^2_k,
  \end{aligned}
\]
where $\chi_k$ is defined in \eqref{eq:homing}, and we set $L_{m,m}=D_{m,m} = V_{m,m}=0$, with $C_D$ and $C_H$ to be determined.
We also introduce an event for a $\gamma$ to be determined
\[
  \mathcal{L}_{m,n}
  =\{ 2D_{j,n} \leq L_{j,n} : m \leq j \leq n-n^{\gamma/4}\}.
\]

\begin{lemma}\label{lem:super}
  Suppose $\sqrt{\eta} \leq \delta/2$.
  For any $j \geq m$, 
  if $|u-\homp_{j}| \geq r_0 \frks_j$
  then the process
  \(
  \mart_{k \wedge T_j}
  \)
  is a $\Pr(\cdot \mid \filt_j)$--supermartingale 
  for $k \geq j$.
\end{lemma}
\begin{proof}
  Provided $k < T_j$ we have $|\homp_k - u| \geq r_0 \frks_{k+1}$. Hence we can use Lemma \ref{lem:homing} and \eqref{eq:homing}
  for $k \geq j$ to conclude if $\chi_k=1$
  \begin{equation*}
    \begin{aligned}
    &\Exp
    \left(
    \exp\bigl(
    \lambda
    |X_{k+1}(\A_{k+1}) + \homp_k - u|
    +\lambda \delta \frks_{k+1}
    - C_H \lambda^2 \frks_{k+1}^2
    \bigr)
    ~\middle\vert~ \filt_k
    \right)
    \\
    &\leq\Exp
    \left(
    \exp\bigl(
    \lambda \frks_{k+1}
    \bigl( \frks_{k+1}^{-1}|X_{k+1}(\A_{k+1}) + \homp_k - u|
    +\delta \bigr) 
    - C_H \lambda^2 \frks_{k+1}^2
    \bigr)
    ~\middle\vert~ \filt_k
    \right) \\
    &\leq \exp( \lambda|\homp_k - u|).
  \end{aligned}
  \end{equation*}
  Furthermore, since $\chi_k=1$
  \begin{align*}
   |\homp_{k+1}-u|
   &\leq 
   |X_{k+1}(\A_{k+1}) + \homp_k - u|
   +|F_{m,k}(\A_{k+1})-\homp_k|\\
   &\leq
   |X_{k+1}(\A_{k+1}) + \homp_k - u|+\sqrt{\eta}\frks_{k+1}.
  \end{align*}
  Hence provided $\sqrt{\eta} \leq \delta/2$ and $\chi_k = 1$
  \[
    \Exp
    \left(
    \exp\bigl(
    \lambda
    |\homp_{k+1}-u|
    +\tfrac12\lambda \delta \frks_{k+1}
    - C_H \lambda^2 \frks_{k+1}^2
    \bigr)
    ~\middle\vert~ \filt_k
    \right)
    \leq  \exp( \lambda|\homp_k - u|).
  \]
  If $\chi_k = 0$ then $\A_{k+1} = \A_k$ and so
  \[
    |\homp_{k+1}-u| 
    \leq
    |\homp_{k}-u| + |X_{k+1}(\A_k)|.
  \]
  Thus we have when $\chi_k=0$, possibly changing $C_H$
  \[
    \Exp
    \left(
    \exp\bigl(
    \lambda
    |\homp_{k+1}-u|
    - C_D\lambda \frks_{k+1}
    - C_H \lambda^2 \frks_{k+1}^2
    \bigr)
    ~\middle\vert~ \filt_k
    \right)
    \leq  \exp( \lambda|\homp_k - u|).
  \]
  Thus, in all cases, we can write
  \begin{multline*}
      \Exp
      \left(
      \exp\bigl(
      \lambda
      |\homp_{k+1}-u|
      +(\tfrac12\delta\chi_k- C_D(1-\chi_k))\lambda \frks_{k+1}
      - C_H \lambda^2 \frks_{k+1}^2
      \bigr)
      ~\middle\vert~ \filt_k
      \right)\\
      \leq  \exp( \lambda|\homp_k - u|).
  \end{multline*}
  
  Increasing $C_D$ to absorb the term $\tfrac12 \delta \chi_k$,
  this implies
  $\Exp( \mart_{k+1 \wedge T_j} \mid \filt_k) \leq \mart_{k \wedge T_j},$ completing the proof.
\end{proof}

Our goal is to show that at time $n$, the event $|u-\homp_n| \leq \frks_n^{1-\gamma}$ is likely.

\begin{lemma}
  For $\beta \in (\frac12, 1],$
  there exist constants (depending only on $\beta$) $n_0 > 0$ and $a \in (0,1)$ such that for any $n \geq n_0$ the sequence $L_{j,n}^2/V_{j,n}$ is monotone increasing in $j \in [1,an]$.
  \label{lem:Gammamonotonicity}
\end{lemma}
\begin{proof}
  We can write
  \[
    \frac{L_{j+1,n}^2}{V_{j+1,n}}
    =
    \frac{(L_{j,n} - \frks_{j+1})^2}{V_{j,n}-\frks_{j+1}}.
  \]
  This being larger than $L_{j,n}^2/V_{j,n}$ is equivalent to the inequality
  \[
    \frks_{j+1}( V_{j,n} + L_{j,n}^2) \geq 2 L_{j,n} V_{j,n}.
  \]
  Thus it suffices to check that 
  \begin{equation}\label{eq:slv}
    \frks_{j+1} L_{j,n} \geq 2 V_{j,n}, \qquad \text{for }\,j \in [1,an].
  \end{equation}

  As we have $\frks_k \sim c k^{-\beta}\omega(k)$
  we can lower bound
  \[
     L_{j,n}
     \geq
     L_{j,\intpart{j/a}}
     \sim c\sum_{k=j}^{\intpart{j/a}} \frac{\omega(k)}{k^\beta} 
     \sim 
     \begin{cases}
     c^\prime(b) a^{\beta - 1} j^{1-\beta} \omega(j), & \beta \in (\tfrac12, 1); \\
     c \log(1/a) \omega(j), & \beta = 1.
     \end{cases}
  \]
  On the other hand
  \[
    V_{j,n} \leq \sum_{k=j+1}^\infty \frks_k^2 \sim 
    \begin{cases}
    c^{\prime\prime}(b) j^{1-2\beta} \omega^2(j), & \beta \in (\tfrac12, 1); \\
    c j^{-1}\omega^2(j), & \beta = 1.
    \end{cases}
  \]
  Hence, for $\beta = 1$, if $c\log(1/a) > 2$ and $j \geq j_0$ for some $j_0$ sufficiently large, then \eqref{eq:slv} holds, and similarly for $\beta \in (1/2, 1)$, if $a > 0$ is sufficiently small depending on $\beta$.
  On the other hand, for $j \leq j_0$ by choosing $n_0$ sufficiently large, we ensure \eqref{eq:slv} holds since $L_{j_0, n_0}$ diverges while the other terms are bounded.
\end{proof}

We will need the following estimates on $L$ and $V$ going forward.
\begin{equation}
  V_{j,n} 
  \asymp
  \begin{cases}
    j^{1-2\beta}\omega^2(j),
    &\text{if } j < a n ; \\
    (n-j)n^{-2\beta} \omega^2(n),
    &\text{if } j \geq a n. \\
  \end{cases}
  \label{eq:varvj}
\end{equation}
As for the sum of standard deviations
\begin{equation}
  L_{j,n}
  \asymp
  (n-j)n^{-\beta} \omega(n)
  \quad\text{if}\quad j \geq a n.
  \label{eq:stdlj}
\end{equation}

\begin{lemma}\label{lem:homingcenter}
  For any $\gamma \in (0,\tfrac14)$
  there is a $c>0$ so that for all $n\geq m/a$ (with $a$ the constant from Lemma \ref{lem:Gammamonotonicity}) such that
  $L_{m,n} \geq \max\{4|u|,\sqrt{V_{m,n}}\}$
  \[
    \Pr( \{
      |\homp_n-u| > \frks_n^{1-\gamma}
    \}
    \cap \mathcal{L}_{m,n}
    ) \leq \exp(-c \min\{ L_{m,n}^2/V_{m,n}, n^{\gamma/4}\}).
  \]
\end{lemma}

\begin{proof}
  We start by bounding the probability that $T_m > n$.
  Using optional stopping and Lemma \ref{lem:super},
  \[
    \Exp[ 
      \one[ T_m > n ]
      \one[ \mathcal{L}_{m,n}]
      \mart_{n \wedge T_m}
    ]
    \leq
    \Exp[ 
      \mart_{n \wedge T_m}
    ]
    \leq \mart_m.
  \]
  On the other hand if $T_m > n$
  then on $\mathcal{L}_{m,n}$
  it follows that
  \(
  \mart_{n \wedge T_m}
  \geq e^{\lambda L_{m,n}/2-\lambda^2 V_{m,n}},
  \)
  and therefore
  \[
    \Pr(\{T_m > n\}\cap \mathcal{L}_{m,n}) 
    \leq \mart_m e^{-\lambda L_{m,n}/2 + \lambda^2 V_{m,n}}
    = e^{\lambda( |u| - L_{m,n}/2) + \lambda^2 V_{m,n}}.
  \]
  Thus as $L_{m,n} \geq 4|u|$, optimizing in taking $\lambda = L_{m,n}/(8V_{m,n})$ results in
  \[
    \Pr(\{T_m > n\} \cap \mathcal{L}_{m,n})
    \leq \exp( -L_{m,n}^2/(128 V_{m,n})).
  \]

  On the event $T_m \leq n$ there is a last time $m \leq J \leq n$ at which $|\homp_J-u| < r_0 \frks_{J+1}$. If $J=n$ then $|\homp_n-u| \leq \frks_n^{1-\gamma}$ (for all $n$ sufficiently large with respect to $r_0$).  Thus we have
  \[
    \begin{aligned}
    &\{|\homp_n-u| > \frks_n^{1-\gamma}\} \cap \{T_m \leq n\}
    \subseteq
    \bigcup_{j=m+1}^{n}
    E_j,
    \quad\text{where},\\
    &E_j \coloneqq
    \{|\homp_{j-1}-u| < r_0 \frks_{j-1}\}
    \cap
    \{T_j > n\}
    \cap
    \{|\homp_n-u| > \frks_n^{1-\gamma}\}.
    \end{aligned}
  \]
Once more using optional stopping and Lemma \ref{lem:super},
on the event $\{|\homp_{j}-u| \geq r_0 \frks_{j}\}$

  \begin{multline*}
    \Exp[ 
      \one[ E_j ]
      \one[ \mathcal{L}_{m,n}]
      \mart_{n \wedge T_j}
      \mid \filt_j
    ]\\
    \leq
    \Exp[ 
      \mart_{j \wedge T_j}
      \mid \filt_j
    ]
    =
    \exp\bigl(
    \lambda \bigl(
    |\homp_j-u|
    +
    L_{m,j}-D_{m,j}\bigr)
    -\lambda^2 V_{m,j}
    \bigr).
  \end{multline*}
  Multiplying by the indicator of $\{|\homp_{j-1}-u| < r_0 \frks_j\}$ and taking conditional expectation, we conclude
  \begin{multline*}
    \Exp[ 
      \one[ E_j ]
      \one[ \mathcal{L}_{m,n}]
      \mart_{n \wedge T_j}
      \mid \filt_{j-1}
    ]\\
    \leq
    \Exp
    [
    \exp\bigl(
    \lambda \bigl(
    |\homp_j-\homp_{j-1}|
    +r_0\frks_j
    +
    L_{m,j}
    -D_{m,j}
    \bigr)
    -\lambda^2 V_{m,j}
    \bigr)
  \mid \filt_{j-1}].
  \end{multline*}
  As $|\homp_j-\homp_{j-1}|$ is bounded (in law) by the sum $\frks_{j}(|Z_1|+|Z_2|+\tfrac12\delta\chi_{j-1})$ for two standard complex normals (which need not be independent), we conclude for some absolute constant $C>0$
  \[
    \Exp[
      \exp\bigl(
      \lambda
      |\homp_j-\homp_{j-1}|
      \bigr)
    \mid \filt_{j-1}]
    \leq \exp( C(\lambda\frks_j+\lambda^2\frks_j^2)).
  \]
  On the other hand on the event $E_j \cap \mathcal{L}_{m,n}$ for $j \leq n-n^{\gamma/4}$
  \[
    \begin{aligned}
    \mart_{n \wedge T_j}
    &\geq
    \exp\bigl(
    \lambda \frks_n^{1-\gamma} + \lambda (L_{m,n}-D_{m,n}) - \lambda^2 V_{m,n}
    \bigr) \\
    &\geq
    \exp\bigl(
    \lambda \frks_n^{1-\gamma} + \lambda (L_{m,j}-D_{m,j}) + \tfrac12\lambda L_{j,n} - \lambda^2 V_{m,n}
    \bigr). \\
    \end{aligned}
  \]
  Putting everything together, we have for any $\lambda \geq 0$ and $j \leq n-n^{\gamma/4}$
  \begin{equation}\label{eq:ej}
    \Pr( E_j \cap \mathcal{L}_{m,n})
    \leq
    \exp\bigl(
    -\lambda(L_{j,n}/2 + c\frks_n^{1-\gamma}\bigr)
    +\lambda C\frks_{j-1}
    +\lambda^2 C \frks_j^2
    +\lambda^2 V_{j,n}
    \bigr),
  \end{equation}
  with $c=1$. We shall show this also holds for $j \geq n-n^{\gamma/4}$ by decreasing $c$.
  If we set $j_0 = \intpart{n-n^{\gamma/4}}$, we can bound on the event $E_j$
  \[
    \mart_{n \wedge T_j}
    \geq
    \exp\bigl(
    \lambda \frks_n^{1-\gamma} + \lambda (L_{m,j}-D_{m,j} - C L_{j_0,n}) - \lambda^2 V_{m,n}
    \bigr).
  \]
  Using \eqref{eq:stdlj}, $L_{j_0,n} \leq n^{\gamma/4-\beta}\omega(n) \leq \frks_n^{1-\gamma+\epsilon}$ for any small $\epsilon >0$ and all $n$ sufficiently large, and so, by decreasing $c$, we can just absorb the $L_{j_0,n}$ term by the $c\frks_n^{1-\gamma}$ term in \eqref{eq:ej}.

  For the terms $j \geq a n$, this leads us to the bound (using \eqref{eq:varvj} and \eqref{eq:stdlj})
  \[
    \Pr(E_j \cap \mathcal{L}_{m,n})
    \leq
    \begin{cases}
    \exp( -c(n-j) ), 
    &\text{if } an \leq j < n-(n^{\beta}/\omega(n))^{\gamma}; \\
    \exp( -c(n^{\beta}/\omega(n))^{\gamma} ), 
    &\text{if } j \geq n-(n^{\beta}/\omega(n))^{\gamma}. \\
    \end{cases}
  \]
  Hence we have for $n$ large enough
  \[
    \sum_{j=\intpart{a n}}^n \Pr(E_j \cap \mathcal{L}_{m,n}) \lesssim \exp(-n^{\gamma/4}). 
  \]
  As for $m+1 \leq j \leq an$, once again we use \eqref{eq:ej} and note that all of the $\frks_k$ terms can be absorbed in the $L_{j,n}$ and $V_{j,n}$ terms, and hence optimizing in $\lambda,$ 
  \[
    \sum_{j=m+1}^{\intpart{a n}}\Pr(E_j \cap \mathcal{L}_{m,n}) 
    \leq \sum_{j=m+1}^{\intpart{a n}} \exp( -cL_{j,n}^2/V_{j,n}).
  \]
  We further break the argument into the cases of $\beta < 1 - \gamma/4$ and $\beta \ge 1 - \gamma/4$. 
  When $\beta < 1 - \gamma/4$, then uniformly over $m+1\leq j \leq an$
  \[
    L_{j,n} \gtrsim n^{1-\beta}\omega(n).
  \]
  As the variance $V_{j,n} \leq V_{0,\infty} < \infty,$ we conclude for some other $c'$
  \[
    \sum_{j=m+1}^{\intpart{a n}}\Pr(E_j \cap \mathcal{L}_{m,n}) \leq \exp( -c' n^{2(1-\beta)}\omega^2(n)) \leq \exp( -c n^{\gamma/4}).
  \]
  As for the case $\beta > 1 - \gamma/4,$ 
  we have
  \(
    V_{j,n} \asymp j^{1-2\beta}\omega^2(j).
  \)
  We further divide $j$ to $j < \Gamma_{m,n} \coloneqq L_{m,n}^2/V_{m,n}$ and $j \geq \Gamma_{m,n}$.
  For the terms $j$ with $j \geq \Gamma_{m,n}$, we use that
  \[
    L_{j,n} \geq L_{j,\intpart{j/a}} \gtrsim \sum_{k=j+1}^{\intpart{j/a}} \frac{\omega(k)}{k^\beta} \gtrsim j^{1-\beta} \omega(j).
  \]
  Hence for some $c' > 0$
  \[
    \sum_{j=\intpart{\Gamma_{m,n}}}^{\intpart{an}} \exp(-cL_{j,n}^2/V_{j,n}) 
    \leq
    \sum_{j=\intpart{\Gamma_{m,n}}}^{\intpart{an}} \exp(-c'j) \lesssim \exp(-c' \Gamma_{m,n}).
  \]
  For $m+1 \leq j \leq \min\{\Gamma_{m,n},an\}$, we have from Lemma \ref{lem:Gammamonotonicity} that 
  \[
    \sum_{j=m+1}^{\intpart{\min\{\Gamma_{m,n},an\}}}
    \exp(-cL_{j,n}^2/V_{j,n})
    \leq \Gamma_{m,n} \exp(-c \Gamma_{m,n})
    \leq \exp(-c' \Gamma_{m,n})
  \]
  for some smaller constant $c'$.
  Putting everything together, we have shown there is a constant $c>0$ so that
  \[
    \Pr(\{|\homp_n-u| > \frks_n^{1-\gamma}\} \cap \{T_m \leq n\} \cap \mathcal{L}_{m,n})
    \leq
    \exp(-c \min\{ \Gamma_{m,n}, n^{\gamma/4}\}).
  \]
  This completes the proof as $\Pr(\{T_m > n\} \cap \mathcal{L}_{m,n})$ is even smaller.
\end{proof}

We next need to bound the probability of the event $\mathcal{L}_{m,n}^c.$  
\begin{lemma}
  For any $\gamma \in (0,\tfrac 14)$ and for all $\eta > 0$ sufficiently small
  \[
  \Pr( \mathcal{L}_{m,n}^c)
  \leq
  \exp(-c \min\{ L_{m,n}^2/V_{m,n}, n^{\gamma/4}\}).
  \]
  \label{lem:unlucky}
\end{lemma}
\begin{proof}
  We begin by bounding (using $1-2^{-L} \geq \tfrac 12$)
  \[
    \begin{aligned}
      D_{j,n}
      &\lesssim \sum_{k=j+1}^n (1-\chi_{k-1})\frks_k \\
      &\leq \frac{1}{\sqrt{\eta}}\sum_{k=j+1}^n |F_{m,k-1}(\A_{k-1}) -F_{m,{k-1}}(\A_{k-1}+\mathfrak{t}_{k})| \\ 
      &\leq \frac{1}{\sqrt{\eta}}
      \sum_{k=j+1}^n
      \sum_{\ell=m+1}^{k-1}
      |X_{\ell}(\A_{k-1}) -X_{\ell}(\A_{k-1}+\mathfrak{t}_{k})| \\
      &=
      \frac{1}{\sqrt{\eta}}
      \sum_{\ell=m+1}^{n}
      \sum_{k=(j\vee \ell)+1}^n
      |X_{\ell}(\A_{k-1}) -X_{\ell}(\A_{k-1}+\mathfrak{t}_{k})| \\
      &\leq
      \frac{2}{\sqrt{\eta}}
      \sum_{\ell=m+1}^{n}
      \sum_{k=(j\vee \ell)+1}^n
      \Delta_{\ell}(\A_{\ell},4) 2^{-L(k-\ell-1)/2} \\
      &\leq
      \frac{2}{\sqrt{\eta}}
      \biggl(
      \sum_{\ell=m+1}^{j}
      \Delta_{\ell}(\A_{\ell},4) 2^{-L(j-\ell)/2} 
      +
      \sum_{\ell=j+1}^{n}
      \Delta_{\ell}(\A_{\ell},4) 
      \biggr).
    \end{aligned}
  \]
Define
\[
  \mathfrak{D}_{m,n}^0(\alpha) \coloneqq \sum_{k=m+1}^n \Delta_k(\A_k,\alpha),\quad
  \mathfrak{D}_{m,n}^1(\beta,\alpha) \coloneqq \sum_{k=m+1}^n \Delta_k(\A_k,\alpha)2^{-L(n-k)/2}.
\]
Then we have shown
\[
      D_{j,n}
      \lesssim
      \frac{1}{\sqrt{\eta}}
      \Bigl(   \mathfrak{D}_{m,j}^1(4) +   \mathfrak{D}_{j,n}^0(4)  \Bigr).
\]
Using Lemma \ref{lem:Dmnbound} below, it follows that there is a constant $C>0$ so that for all $\eta >0$ and all $t>0$
\[
      \Pr(
      D_{j,n}
      \geq 
      C\sqrt{\eta}L_{j,n} + t)
      \leq \exp(-t^2/(C\eta^2 V_{j,n})).
\]
Hence for $\eta$ sufficiently small that $C\sqrt{\eta} \leq \tfrac 14$
\[
  \Pr( \mathcal{L}_{m,n}^c)
  \leq \sum_{j=m+1}^{n-n^{\gamma/4}} 
  \exp( -c L_{j,n}^2/V_{j,n} )
\]
Using the same bounds for this sum as in the proof of Lemma \ref{lem:homingcenter},
we conclude there is a constant $c>0$ such that
\[
  \Pr( \mathcal{L}_{m,n}^c)
  \leq
  \exp(-c \min\{ L_{m,n}^2/V_{m,n}, n^{\gamma/4}\}).
\]
\end{proof}

\begin{lemma}
  For any $\alpha \geq 1$ 
  there is a constant $C > 0$
  so that for all $\eta >0$ and
  for all $j \leq k$
  \[
    \|\mathfrak{D}_{j,k}^1(\alpha)\|_{\psi_2}
    \leq C\eta\frks_k
    \quad
    \text{and}
    \quad
    \|\mathfrak{D}_{j,k}^0(\alpha)- \Exp \mathfrak{D}_{j,k}^0(\alpha)\|^2_{\psi_2}
    \leq C\eta^2 V_{j,k}.
  \]
  Furthermore,
  \(
    \Exp \mathfrak{D}_{j,k}^0(\alpha)
    \leq C\eta L_{j,k}.
  \)
  \label{lem:Dmnbound}
\end{lemma}
\begin{proof}[Proof of Lemma \ref{lem:Dmnbound}]
  The first claim follows from Lemma~\ref{lem:homingdecoupling} and the triangle inequality, since
  \[
    \|\mathfrak{D}_{j,k}^1(\alpha)\|_{\psi_2}
    \leq C\eta \sum_{\ell=j+1}^k \frks_\ell 2^{-L(k-\ell)/2}
    \lesssim \eta \frks_k.
  \]
  The third claim follows immediately from Lemma \ref{lem:homingdecoupling}.
  For the second claim, using \cite[Proposition 2.6.1]{Vershynin}
  there is an absolute constant $C>0$
  \[
    \|\mathfrak{D}_{j,k}^0(\alpha)- \Exp \mathfrak{D}_{j,k}^0(\alpha)\|^2_{\psi_2}
    \leq C
    \sum_{\ell=j+1}^k \|\Delta_\ell(\A_\ell,\alpha) - \Exp \Delta_\ell(\A_\ell,\alpha)\|^2_{\psi_2}
    \lesssim
    \eta^2 \sum_{\ell=j+1}^k \frks_\ell^2.
  \]
\end{proof}

\begin{lemma}
  The H\"older constant of the GAF $F_{n,m}$ near $\A_n$ satisfies
  \[
    \Delta_{m,n}(\A_n,\alpha)
    \leq
    \mathfrak{D}_{m,n}^1(4\alpha).
  \]
  \label{lem:windowbound}
\end{lemma}
\begin{proof}
  We begin with the identity
  \[
    \begin{aligned}
      F_{m,n}(\psi,r)
      &=
      \sum_{j=m+1}^n
      X_j(\psi,r). 
    \end{aligned}
  \]
  Thus the H\"older constant of $F_{m,n}$ over any set is bounded by the sum of the H\"older constants of the $X_j.$  The $\Delta_j$ are these H\"older constants (in fact they are H\"older constants on an even larger set as $\mathcal{W}_n(\A_n,\alpha) \subseteq \mathcal{W}_j(\A_j,4\alpha)$) \emph{after} rescaling by $2^{-L(n-j)/2},$ and so
  \[
    \Delta_{m,n}(\A_n,\alpha)
    \leq \sum_{j=m+1}^n \Delta_j(\A_j,4\alpha) 2^{-L(n-j)/2}
    =\mathfrak{D}_{m,n}^1(4\alpha).
  \]
\end{proof}

We now conclude the proof of the main result of this section.
\begin{proof}[Proof of Proposition \ref{prop:homing}]
  From Lemma \ref{lem:unlucky} and Lemma \ref{lem:homingcenter}
  \[
    \Pr( |\homp_n - u| \geq \frks_n^{1-\gamma}) 
    \leq 2\exp(-c \min\{ L_{m,n}^2/V_{m,n}, n^{\gamma/4}\}).
  \]
  From Lemma \ref{lem:windowbound}, to complete the proof it suffices to bound the probability 
  \(
  \mathfrak{D}_{m,n}^1(\alpha) \geq \frks_n^{1-\gamma}.
  \)
  Using Lemma \ref{lem:Dmnbound}, there is a $c>0$ so that for all $\eta > 0$
  \[
    \Pr(\mathfrak{D}_{m,n}^1(\alpha) \geq \frks_n^{1-\gamma})
    \leq 2\exp( -c \frks_n^{-2\gamma}/\eta^2).
  \]
  Noticing that $\frks_n^{-2\gamma} \geq n^{\gamma/4}$, this completes the proof.

\end{proof}

\section{Covering small disks}\label{sec:branching}

We split the series \eqref{eq:powerseries} into three parts, the homing, spawning, and covering portions of the GAF.  These are given by,
\[
  F = F_{0,k_h} + F_{k_h,k_s} + F_{k_s,\infty},
\]
where $k_s \geq k_h$ are parameters to be determined.

We will use the shorthand in this section $F_h$, $F_s$ and $F_{c}$ for these 3 GAFs respectively.
We also introduce a radial parameter $r_c = 1-2^{-LM}$ for $M = k_s^2$ which will be the radius at which we perform the coverage argument.
We also introduce $\Sigma_h$, $\Sigma_s$ and $\Sigma_c$ for the standard deviations of $F_h$, $F_s$, and $F_c$ at $r_c$, and we let $\Sigma_{c'}$ be the standard deviation of $F_c'$ at $r_c.$

Let $\Theta$ be a uniform random angle chosen independently from the arc $\A_{k_h} + [-2^{- L k_h},2^{- L k_h}]$.
Introduce two events, for $\Delta$ and $\Delta_c$ to be determined and for any $\theta$ in the arc
\[
  \begin{aligned}
  &\mathcal{E}_s(\theta) \coloneqq
  \{ |F_s(\theta,r_c)| \leq \Delta\}, \quad \text{and}\\
  &\mathcal{E}_c(\theta) \coloneqq
  \{ |F(\theta,r_c)-u| \leq \Delta_c, |{F'_c}(\theta,r_c)| \in [\Sigma_{c'},2\Sigma_{c'}] \}.
  \end{aligned}
\]
We let $\mathscr{F}_h$ be the $\sigma$-algebra generated by $F_h$.

We will work on the event
\begin{equation}\label{eq:Eh}
  \mathcal{E}_h \coloneqq \{
  |F_h(\theta,r_c)-u| \leq \Delta,
  \quad \text{for all } \theta \in \A_{k_h} + [-2^{- L k_h},2^{- L k_h}]
  \}.
\end{equation}
We now bound below, conditionally on $\mathscr{F}_h$, the probability that there exists a $\theta$ such that  $\mathcal{E}_c(\theta)$ occurs.

Define, for a pair $(Y,Z)$ of standard complex normals, with $|\Exp Y \overline{Z}| = \rho$
\[
  p(x,\delta,\rho)
  \coloneqq
  \Pr( |Y+x| \leq \delta, |Z| \in [1,2]).
\]
This is connected to the probability of $\mathcal{E}_c(j)$ by
\[
  \Pr( \mathcal{E}_c(\theta) ~|~ \mathscr{F}_s)
  = p\Bigl( \tfrac{F_h+F_s-u}{\Sigma_c}, \tfrac{\Delta_c}{\Sigma_c}, \rho_{c,c'}(\theta,\theta) \Bigr)
  \eqqcolon p_c(\theta)
\]
where 
\[
  \rho_{\alpha,\alpha^\dag}(\theta_1,\theta_2)
  \coloneqq
  \Big(\Sigma_{\alpha}
  \Sigma_{\alpha^\dag}\Big)^{-1}
  |\Exp\bigl[{F_\alpha}(\theta_1,r_c)\overline{{F_{\alpha^\dag}}(\theta_2,r_c)}\bigr]|.
\]
We note that by rotation invariance, $\rho_{\alpha ,\alpha^\dag}(\theta,\theta)$ does not depend on $\theta,$ and so we will drop the dependence.
\begin{lemma}
  For all $|x| \leq 1$, $\rho,\delta \leq \tfrac12$,
  there is a positive constant $c$,
  \[
  p(x,\delta,\rho)
  \geq c\delta^2.
  \]
  \label{lem:1point}
\end{lemma}
\begin{proof}
This is immediate from the form of the bivariate complex Gaussian density.
\end{proof}

We also need a corresponding computation for $\Pr( \mathcal{E}_c(\theta_1) \cap \mathcal{E}_c(\theta_2) ~|~ \mathscr{F}_s)$.  Define $\rho_{\mathrm{max}}$
\[
  \rho_{\mathrm{max}}(\theta_1,\theta_2)
  \coloneqq
  \max\{
    \rho_{c,c},
    \rho_{c,c'},
    \rho_{c',c'}
  \}(\theta_1,\theta_2).
\]

\begin{lemma}
  There is an $\varepsilon > 0$ so that, if $\rho_{\mathrm{max}} \leq \varepsilon$, $2\Delta,2\Delta_c \leq \Sigma_{c}$,
  and $\rho_{c,c'}\leq \tfrac12$, then
  there is a constant $C>0$ so that
  on the event $\mathcal{E}_s(\theta_1) \cap \mathcal{E}_s(\theta_2) \cap \mathcal{E}_h$ 
  \[
    \Pr( \mathcal{E}_c(\theta_1) \cap \mathcal{E}_c(\theta_2) ~|~ \mathscr{F}_s)
    \leq
    p_c(\theta_1)p_c(\theta_2)
    (1+C \rho_{\mathrm{max}}).
  \]
  \label{lem:2point}
\end{lemma}
\begin{proof}

  We should prove that for $\delta,\rho \leq \tfrac 12$ if $\{(Y_j,Z_j) : j=1,2\}$ are four jointly complex standard Gaussians with $|\Exp Y_j \overline{Z_j}| = \rho$, and with $\rho_{\mathrm{max}}$ the maximum correlation of a $j=1$ variable with a $j=2$ variable, then there is an absolute constant $C>0$ so that
  for any $x_j$ with $|x_j| \leq 1$
    \[
    \Pr( \cap_{j=1,2} \{ |Y_j+x_j| \leq \delta, |Z_j| \in [1,2] \})
    \leq
    p(x_1,\delta,\rho)
    p(x_2,\delta,\rho)
    (1+
    C \rho_{\mathrm{max}}).
  \]
Let $\psi(\omega;\vec\Sigma\,)$ be the density on $\C^4$ of $\{(Y_j,Z_j) : j=1,2\}$ where $\vec\Sigma$ are the correlations in the off-diagonal block.  Then $\log \psi$ is an analytic function of $(\omega,\vec\Sigma)$ in a neighborhood of $\vec\Sigma = \vec 0$.   Moreover at $\vec\Sigma=\vec 0,$ we have
  \[
    \Pr( \cap_{j=1,2} \{ | Y_j+x_j| \leq \delta, |Z_j| \in [1,2] \})
    =
    p(x_1,\delta,\rho)
    p(x_2,\delta,\rho).
  \]
  Under the assumptions of the lemma, $\omega$ is contained within a compact set of $\C^4$, and so the bound follows.
\end{proof}

We will apply the second moment method. We let 
\[
  H_s \coloneqq \Pr(\mathcal{E}_s(\Theta) \mid F)
  \quad\text{and}\quad
  H_c \coloneqq 
  \Pr(\mathcal{E}_s(\Theta) \cap \mathcal{E}_c(\Theta)  \mid F),
\]
which is to say that each of these is the normalized area in the arc $\A_{k_h} + [-2^{-Lk_H},2^{-Lk_h}]$ in which $\mathcal{E}_s$ and $\mathcal{E}_c$ hold, respectively (we condition on all the GAF coefficients).

Let $\mathcal{H}_s$ be the set of $\theta$ so that $\mathcal{E}_s(\theta)$ holds.
Let $\epsilon_b$ to be determined (but sufficiently small that Lemma \ref{lem:2point} holds) and define
for two i.i.d. copies $\Theta_1,\Theta_2$ of $\Theta$
\[
  B = \Pr(\rho_{\mathrm{max}}(\Theta_1,\Theta_2) \geq \epsilon_b) .
\]
To demonstrate that there is a $\theta$ at which $\mathcal{E}_s(\theta)\cap\mathcal{E}_c(\theta)$ holds, it suffices to show that $H_c$ is positive.
So, using the Paley-Zygmund inequality,
\[ 
  \Pr( H_c > 0 ~|~ \mathscr{F}_s) \geq 
  \frac{
    \bigl(\Exp \bigl[H_c ~|~\mathscr{F}_s\bigr]\bigr)^2
  }
  {
  \Exp \bigl[ H_c^2 ~|~ \mathscr{F}_s\bigr]
  }
  .
\]
For the numerator, on the event $\mathcal{E}_h$ and using Lemma \ref{lem:1point}
\[
  \Exp \bigl[H_c ~|~\mathscr{F}_s\bigr] = 
  \Exp p_c(\Theta) \one[ \Theta \in \mathcal{H}_s]
  \geq 
  cH_s \Delta_c^2/\Sigma_c^2.
\]
For the denominator, we bound on the event $\mathcal{E}_h$ (and provided the numerical assumptions in Lemma \ref{lem:2point} hold)
\begin{equation*}
  \begin{aligned}
   \Exp \bigl[H_c^2 ~|~\mathscr{F}_s\bigr] 
   &=  
   \Exp\Bigl( \one[ \Theta_1,\Theta_2 \in \mathcal{H}_s]
   \Pr( \mathcal{E}_c(\Theta_1) \cap \mathcal{E}_c(\Theta_2) ~|~ \mathscr{F}_s\Bigr) \\
   &\leq  
   \Exp\Bigl( \one[ \Theta_1,\Theta_2 \in \mathcal{H}_s]
    p_c(\Theta_1)p_c(\Theta_2)(1+C\epsilon_b)~|~ \mathscr{F}_s\Bigr) 
    +B \\
   &\leq
   \Bigl(\Exp \bigl[H_c ~|~\mathscr{F}_s\bigr]\Bigr)^2(1+C\epsilon_b) + B.
  \end{aligned}
\end{equation*}

So for some constant $C>0$, on the event $\mathcal{E}_h$ 
\begin{equation}\label{eq:Hc}
  \Pr( H_c > 0 ~|~ \mathscr{F}_s)
  \geq 
  \biggl(1 + C \epsilon_b 
  +
  \tfrac{CB}{ (H_s  \Delta_c^2/\Sigma_c^2)^2}
  \biggr)^{-1}.
\end{equation}
All of this holds under the assumptions  $2\Delta,2\Delta_c \leq \Sigma_{c}$,
  and if $\rho_{c,c'}\leq \tfrac12$.
We now give high probability estimates for these quantities, which are based on variance and covariance estimates of different parts of the GAF.


\paragraph{Computation of the 1-point variance/covariances.}

Under the assumption $\frks_k^2 \sim k^{-2\beta}\omega^2(k)$ is regularly varying, 
$M = k_s^2$,  $k_s/k_h \leq 2$ (we shall take $k_s -k_h$ logarithmic in $k_h$), and $r_c = 1-2^{-LM}$, 
\begin{equation}
  \begin{aligned}
  \Sigma_s^2 
  &= \Exp |F_s(r_c)|^2 
  =  \sum_{n=2^{L (k_h+1)}}^{2^{L (k_s+1)}-1}
  a_n^2 r_c^{2n}
  \asymp
  \sum_{k=k_h}^{k_s} \frks_k^2
  \lesssim_L {k_h^{-2\beta}} \omega^2(k_h)(k_s-k_h), \\
  \Sigma_c^2 
  &= \Exp |F_c(r_c)|^2 
  =  
  \sum_{n=2^{L (k_s+1)}}^{\infty}
  a_n^2 r_c^{2n}
  \asymp_L
  \sum_{k=k_s}^{k_s^2}
  \frks_k^2
  \asymp_L k_h^{-2\beta+1} \omega^2(k_h),
  \\
  \Sigma_{c'}^2
  &= \Exp |F'_c(r_c)|^2 
  =  \sum_{n=2^{L (k_s+1)}}^{\infty}
  n^2 a_n^2 r_c^{2n-2}
  \asymp_L 
  2^{2LM} M^{-2\beta}\omega^2(M).
\end{aligned}
  \label{eq:Sigma_approx}
\end{equation}

We also need an estimate on the covariance $\rho_{c,c'}$.  Note that we have
\[
  \Sigma_{c'}
  \Sigma_{c}
  \rho_{c,c'}
  =
  \big|\Exp\big[ F'_c(r_c)\overline{F_c(r_c)}\big]\big|
  =  
  \sum_{n=2^{L (k_s+1)}}^{\infty}
  n a_n^2 r_c^{2n-1}
  \asymp_L 
  2^{LM} M^{-2\beta}\omega^2(M).
\]
Hence for any $\epsilon >0$
\begin{equation}\label{eq:rhocc'}
  \rho_{c,c'}
  \lesssim_L 
  {M^{-\beta}}{k_h^{\beta-\tfrac12}}
  \tfrac{\omega(M)}
  {\omega(k_h)}
  \lesssim_L k_h^{-\beta-\tfrac12 + \epsilon}.
\end{equation}

\paragraph{Computation of the 2-point squared correlations.}

Recall that the \\ angles $\Theta_j$ are taken uniformly from an arc, 
which after rotation, is \[[-2^{-L k_h}, 2^{-L k_h}].\]
Now when we compute correlations, we have
\[
  \Sigma_c^2 \rho_{c,c}(\Theta_1,\Theta_2)
  =
  \sum_{n=2^{L (k_s+1)}}^{\infty}
  a_n^2 r_c^{2n}
  e(n(\Theta_1-\Theta_2)).
\]
Hence evaluating the variance, we arrive at
\[
  \Sigma_c^4 \Exp|\rho_{c,c}(\Theta_1,\Theta_2)|^2
  =
  \sum_{n_1,n_2=2^{L (k_s+1)}}^{\infty}
  a_{n_1}^2
  a_{n_2}^2
  r_c^{2(n_1+n_2)}
  \sinc^2((n_1-n_2)2^{-Lk_h}).
\]
Fixing $n_1$ and summing on $n_2 \ge n_1$, we have from the bound $|\sinc(x)| \leq (1 \wedge |x|^{-1})$
\[
  \begin{aligned}
  \sum_{n_2 = n_1}^{\infty}
  &a_{n_2}^2
  r_c^{2(n_1+n_2)}
  \sinc^2((n_1-n_2)2^{-Lk_h})
  \lesssim
  \sum_{n_2 = n_1}^{\infty}
  a_{n_2}^2
  (1 \wedge 2^{Lk_h}/(n_2-n_1))^2
  \\
  &\lesssim
  \sum_{n_2 = n_1}^{\infty}
  (1/n_2) (\log n_2)^{-2\beta} \omega^2(\log n_2)
  (1 \wedge 2^{Lk_h}/(n_2-n_1))^2 \\
  &=
  \sum_{n_2=n_1}^{n_1 + 2^{Lk_h}} \cdot + \sum_{n_2 = n_1 + 2^{Lk_h} + 1}^\infty \cdot
  \lesssim
  2^{Lk_h}
  (1/n_1) (\log n_1)^{-2\beta}\omega^2(\log n_1)
  \sum_{t=1}^{\infty}
  \frac{1}{t^2} \\
  &\lesssim
  2^{Lk_h}
  a_{n_1}^2.
  \end{aligned}
\]
Thus we conclude
\[
  \Sigma_c^4 \Exp|\rho_{c,c}(\Theta_1,\Theta_2)|^2
  \lesssim
  2^{Lk_h}
  \sum_{n_1=2^{L (k_s+1)}}^{\infty}
  a_{n_1}^4
  \lesssim
  2^{L(k_h-k_s)}
  (k_h)^{-4\beta}\omega^4(k_h).
\]
In a similar fashion,
we obtain
\[
  \begin{aligned}
  &\Sigma_c^2 \Sigma_{c'}^2 \Exp|\rho_{c,c'}(\Theta_1,\Theta_2)|^2
  \lesssim
  2^{Lk_h}
  \sum_{n_1=2^{L (k_s+1)}}^{\infty}
  n_1^2
  a_{n_1}^4
  e^{-n_1 2^{-LM}}\\
  &\lesssim
  2^{L(M+k_h)}
  (M)^{-4\beta}\omega^4(M), \\
  &\Sigma_{c'}^4 \Exp|\rho_{c',c'}(\Theta_1,\Theta_2)|^2
  \lesssim
  2^{Lk_h}
  \sum_{n_1=2^{L (k_s+1)}}^{\infty}
  n_1^4
  a_{n_1}^4
  e^{-n_1 2^{-LM}}\\
  &\lesssim
  2^{L(3M+k_h)}
  (M)^{-4\beta}\omega^4(M). \\
  \end{aligned}
\]
All of this implies that the worst (largest) term is the 
$\Exp|\rho_{c,c}(\Theta_1,\Theta_2)|^2$
which leads us to the bound
\begin{equation}\label{eq:cor}
\bigl(\Exp_{\Theta_1,\Theta_2}(\rho^2_{\mathrm{max}})\bigr)
  \lesssim
  2^{L(k_h-k_s)}
  {k_h^{-2}}.
\end{equation}

We conclude with showing that:
\begin{lemma}\label{lem:hittingu}
  Let $\Lambda, R > 0$. With $\Delta = (k_h)^{-\beta+1/4}$, $\Delta_c = k_h^{-\Lambda}$ and $k_s = k_h + 4((\Lambda+R)/L) \log_2 k_h$, on the event $\mathcal{E}_h(u)$,
  for all $k_h \geq k_0(\Lambda,R)$ we have
  \[
    \Pr( H_c > 0 ~\mid~ \filt_h)
    \geq 1-k_h^{-R}.
  \]
\end{lemma}
\begin{proof}
  We recall \eqref{eq:Hc}:
  \begin{equation*}
    \Pr( H_c > 0 ~|~ \mathscr{F}_s )
    \geq 
    \biggl(1 + C \epsilon_b 
    +
    \tfrac{CB}{ (H_s  \Delta_c^2/\Sigma_c^2)^2}
    \biggr)^{-1}.
  \end{equation*}

  Let $\epsilon_b = k_h^{-R}/(4C).$
  We lower bound $H_s$ and we upper bound $B$ in probability.

  \paragraph{Lower bound for $H_s$:}
  We have
  \begin{multline*}
      \Pr(H_s < 1/2 \mid \filt_h) \\
      \leq 
      \Pr\Bigl( \Pr \{\Theta : |F_s(\Theta,r_c)| \geq \Delta \mid F\} \geq \tfrac12) \mid \filt_h \Bigr)
      \leq\Pr( |F_s(\Theta,r_c)| \geq \Delta \mid \filt_h).
  \end{multline*}  
  From \eqref{eq:Sigma_approx}, $\Sigma_s \leq k_h^{-\beta+1/8}$ for all $k_h$ sufficiently large.  Therefore
  \[
    \Pr(H_s < 1/2 \mid \filt_h) \leq 2\exp(-k_h^{1/4}).
  \]

  \paragraph{Upper bound for $B$:}
  Using Markov's inequality and
  \eqref{eq:cor},
  \[
    \begin{aligned}
    \Pr\big( B \geq \epsilon_b^{-2} k_h^{-3(\Lambda + R)} \mid \filt_h\big)
    &\leq
    \Exp\big( B \epsilon_b^2 k_h^{3(\Lambda + R)} \mid \filt_h\big)\\
    &\leq 
    \bigl(\Exp_{\Theta_1,\Theta_2}(\rho^2_{\mathrm{max}})\bigr) k_h^{3(\Lambda + R)}
    \lesssim k_h^{-R-2}.
    \end{aligned}
  \]
  Since $\Delta_c^2/\Sigma_c^2 \geq k_h^{-3 \Lambda}$ for all $k_h$ sufficiently large  we conclude that 
  \[
     \Pr( H_c > 0 ~|~ \mathscr{F}_s )
     \geq 1-k_h^{-R}/2
  \]
  on the events bounded above.  This concludes the proof.
\end{proof}

\begin{proposition}\label{prop:cover}
  With $\Delta = (k_h)^{-\beta+\tfrac14}$, on the event $\mathcal{E}_h(u)$
  and for any $\epsilon,R > 0$, we have for all $k_h \geq k_0(\epsilon,R)$,
  \[
    \Pr( F(\D) \supset \D(u, k_h^{-1-2\beta-\epsilon}) ~\mid~ \filt_h)
    \geq 1-k_h^{-R}.
  \]
\end{proposition}

\begin{proof}[Proof of Proposition \ref{prop:cover}]
  Let $\Delta_c = k_h^{-4}$. On the event that $H_c > 0,$ whose probability is at least $1-k_h^{-R}$ by Lemma  \ref{lem:hittingu}, there is a $z_j$ with $|z_j| = r_c = 1-2^{-LM}$ at which
  \[
   |F(z_j)-u| \leq \Delta_c
   \quad\text{and}\quad
   |{F'_c}(z_j)| \geq \Sigma_{c'}.
  \]
  We apply Lemma \ref{lem:cover} with $R=(1-r_c)/2$, and we bound $S$ by $\sup |F'(z)|$ over $|z| \leq (1+r_c)/2 \eqqcolon r_c'$.  We also bound $A$ by $\Sigma_{c'}-\sup |F_h'+F_s'|$ on the same disk.

  Recall from \eqref{eq:Sigma_approx} that $\Sigma_{c'} \asymp 2^{LM}M^{-\beta}\omega(M)$.  The same computation shows that with $r_c'' = \tfrac12(1+r_c')$ we have
  \[
    \begin{aligned}
  &\Exp |F_c'(r_c'')|^2 \lesssim 2^{2LM+4} M^{-2\beta} \omega^2(M), \, \text{ and } \,\,\\
  &\Exp |(F_h'+F_s')(r_c'')|^2 \lesssim 2^{2Lk_s} (k_s)^{-2\beta} \omega^2(k_s).
    \end{aligned}
  \]
  Hence by Lemma \ref{lem:GAFtail}, with probability at least $1-e^{-cM}$
  \[
    \sup_{r_c' \D} |F_c'| \lesssim \sqrt{M}\Sigma_{c'}
    \quad\text{and}\quad
    \sup_{r_c' \D} |F_h'+F_s'| \leq \Sigma_{c'}/2.
  \]
  Hence the image of $F$ contains a ball around $F(z_j)$ of radius at least
  \[
    CA^2R/S \gtrsim \Sigma_{c'} 2^{-LM} / \sqrt{M} \gtrsim M^{-\beta-1/2-\epsilon/4} = k_h^{-1-2\beta-\epsilon/2}.
  \]
  This is bigger than $\Delta_c$ by a factor that tends to infinity with $k_h$, and since $|F(z_j)-u| \leq \Delta_c$, then the claim follows.

\end{proof}

\section{Covering large disks}\label{sec:coverage}

\begin{proof}[Proof of Theorem~\ref{thm:coverage} in the case $\beta<1$]
  We give a simpler proof of the main theorem in the case that $L_{1,n}^2/\log n \to \infty$ as $n\to\infty$.  This includes all the $\beta < 1$ cases but also some $\beta = 1$ cases; to see this in the case $\beta < 1$, note that
  \[
    L_{1,n} \gtrsim \sum_{j=1}^n j^{-\beta} \omega(j) \gtrsim n^{(1-\beta)/2}.
  \]

  It suffices to show that for any $u \in \C$ and any $\delta > 0$
  \[
    \Pr( F(\D) \supset \D(u,1)) \geq 1 - \delta.
  \]
  Now, with some large $n$ to be determined, we pick a $n^{-4}$--net $\{u_q\}$ of $\D(u,1)$ of cardinality $Cn^8$.
  By Proposition \ref{prop:homing} (with $m=0,$ $u=u_q$, $\gamma = \tfrac18$, $\alpha=1$ and all $n$ sufficiently large that $L_{1,n} \geq \max\{2(|u|+1),\sqrt{V_{0,\infty}}\}$), with probability $1-n^{-10}$ there is a $\Theta_q$ so that
  \[
     |F_{0,n}(\Theta_q) - u_q| + \Delta_{0,n}(\Theta_q,1) \leq 2\frks_n^{7/8}.
  \]
  It follows that
  \[
    \sup_{(\theta,r) \in \mathcal{W}_n(\Theta_q,1)} |F_{0,n}(\theta,r) - u_q|
    \leq |F_{0,n}(\Theta_q) - u_q| + 2\Delta_{0,n}(\Theta_q,1) \leq 4\frks_n^{7/8}.
  \]

  Now we apply Proposition \ref{prop:cover} (with parameters $k_h = n$, $u=u_q$, $\epsilon = \tfrac 12$ and $R=10$) and working on the event that everything in previous paragraph held. These apply since $4\frks_n^{7/8} \leq n^{-\beta + 1/4}$ for all $n$ sufficiently large. In conclusion
  \[
    \Pr(  F(\D) \supset \D(u_q,n^{-3/2-2\beta})) \geq 1-2n^{-10}.
  \]
  By a union bound over all $q$, we conclude for all $n$ sufficiently large
  \[
    \Pr(  F(\D) \supset \D(u,1)) \geq 1-2Cn^{-2} \geq 1-\delta/2,
  \]
  which completes the proof.
\end{proof}

\begin{remark}
  In the case of $\beta < 1$, it is even possible to avoid the homing process entirely.  Instead, one can define an intermediate scale $r$ with $1 \leq r \leq n$.  Then, in the first $r$ blocks, one asks the GAF not to grow too much (in supremum).  In the second portion $F_{r,n}$, one can find $2^r$ (almost) independent Gaussians, and these Gaussians have variance $r^{1-2\beta}$. 
  
  The cardinality $2^r$ beats the probability $e^{-c(|u_q|)r^{2\beta -1}}$ of hitting a point $u_q$ in the net, and so from those points $z_q \in \D$ we can still use Proposition \ref{prop:cover}. In fact, one also does not need Proposition \ref{prop:cover}, in that it is not necessary to look around the $z_q$ and one can simply get close to $u_q$ with a ball around $z_q$ itself.
\end{remark}

We now turn to the case $\beta = 1$.

\begin{proposition}
  \label{prop:homingnet}
  For any $\gamma \in (0,\tfrac14)$
  and $\beta=1$, there is an integer $m$ sufficiently large that the following holds. 
  For any $\epsilon > 0$ and for any $u \in \C$, there is a $j_1$ so that for any $j \geq j_1$, with probability $1-\epsilon$:
  \begin{enumerate}
    \item there is a $\nu_j$-net $\{u_q : 1 \leq q \leq 2^{2^j}\}$ of $\D(u,1)$ with $\nu_j \lesssim 2^{2^{-j-1}}$;
    \item there is a collection $\{\Theta_q : 1 \leq q \leq 2^{2^j}\}$ in $[0,1]$ so that
      \[
	\sup_{ (\theta,r) \in \mathcal{W}_{m^j}(\Theta_q,1)} |F_{0,m^j}(\theta, r) - u_q |
	\leq 12 (\frks_{m^j})^{1-\gamma}.
      \]
  \end{enumerate}
\end{proposition}
\begin{proof}
  Let $u$ and $\epsilon > 0$ be given, and $a$ as in Lemma \ref{lem:Gammamonotonicity}.
  Let $m$ be an integer strictly larger than $\max\{1/a,3^{4/\gamma}\}$.  We can further ensure by increasing $m$ if need be that for all $j$ sufficiently large
  \begin{equation}\label{eq:domination}
    \sum_{k=0}^{j} (\frks_{m^{k}})^{1-\gamma} 2^{-L(m^j-m^{k})/2} \leq 2 (\frks_{m^{j+1}})^{1-\gamma}.
  \end{equation}
  We observe the following estimates which we will use repeatedly:
  for some $c_L,c_V$ positive (depending on $m$)
  \begin{equation}
    L_{k,mk} \sim c_L\omega(k)
    \quad\text{and}
    \quad
    V_{k,mk} \sim c_V k^{-1}\omega^2(k).
    \label{eq:LV}
  \end{equation}
  For every $j \in \N$ we let
  \[
    \{ x_{j, q} : 1 \leq q \leq 2^{2^j}\}
  \]
  be an an $\epsilon_j$-net of $\D(0,1)$ with $\epsilon_j \coloneqq c 2^{-2^{j-1}}$ for some constant $c>0$.  We shall pick a $j_0$ large to be determined.  To start, we will require $j_0$ sufficiently large that for all $j \geq j_0.$
  \[
    \epsilon_j \leq \tfrac16 L_{m^j,m^{j+1}}.
  \]
  We also pick a parameter $j_1 \geq j_0$ to be determined. 

  Let $z_j$ and $b_j$ be two converging sequences where $z_j=b_j=0$ for all $j \leq j_0$,
  where $z_j =u$ and $b_j=1$ for all $j \geq j_1$,
  and where
  \[
    |z_{j+1}-z_j| + | b_{j+1}-b_j| \leq \tfrac16 L_{m^j,m^{j+1}}.
  \]
  This is possible because the series $L_{m^{j_0},m^j} \to \infty$ as $j \to \infty.$
  We further define 
  \[
    \{ u_{j,q} : 1 \leq q \leq 2^{2^j}\}
    \quad\text{where}\quad u_{j,q} \coloneqq b_j x_{j,q} + z_j.
  \]
  We shall also require that $j_0$ is large enough that for all $j \geq j_0$,
  \[
    12( \frks_{m^j})^{1-\gamma}
    \leq
    \tfrac16 L_{m^j,m^{j+1}},
    \,
    L_{0,m^{j_0}} \geq \sqrt{ V_{0,m^{j_0}}},
    \,\,
    \text{and}
    \,\,
    L_{m^j,m^{j+1}} \geq \sqrt{ V_{m^j,m^{j+1}}},
  \]
  where the first inequality is possible as its left-hand side decays exponentially in $j$ while the right-hand side satisfies \eqref{eq:LV}.

  We further require that $j_0$ is sufficiently large that
for $c$ the constant from Proposition \ref{prop:homing} (with $\alpha = 4$)
  \[
    \exp(-c \min\{ L_{0,m^{j_0}}^2/V_{0,m^{j_0}}, (m^{j_0})^{\gamma/4}\}) \leq \epsilon/3
  \]
  and
  \[
    \sum_{j=j_0}^\infty 2^{2^{j+1}}\exp(-c \min\{ L_{m^j,m^{j+1}}^2/V_{m^{j},m^{j+1}}, (m^{j})^{\gamma/4}\}) \leq \epsilon/3.
  \]
  The second inequality follows since
  for $j_0$ large enough, using \eqref{eq:LV} and $m \ge 3^{4/\gamma}$,
  \[
    \sum_{j=j_0}^\infty 2^{2^{j+1}}\exp(-c \min\{ L_{m^j,m^{j+1}}^2/V_{m^{j},m^{j+1}}, (m^{j})^{\gamma/4}\}) 
    \leq
    \sum_{j=j_0}^\infty 2^{2^{j+1}}\exp(- 3^j ).
  \]

  We inductively define a family of angles $\{\theta_{j,q} : j\geq j_0, 1 \leq q \leq 2^{2^j}\}$ by the following rule.
  We first define for all $1 \leq q \leq 2^{2^{j_0}}$
  (recalling the definition of the H\"older constant $\Delta$ in \eqref{eq:Delta})
  \[
    \theta_{j_0,q} = \operatorname{argmin}\{|F_{0,m^{j_0}}(\theta)| + \Delta_{0,m^{j_0}}(\theta,4): \theta \in [0,2\pi]\}. 
  \]
  For larger $j_0 \leq j < j_1$, we define $\theta_{j+1,q}$ for every $1 \leq q \leq 2^{2^{j+1}}$, by first finding, for each $q$ a $q'=q'(q,j)$ with $1 \leq q' \leq 2^{2^j}$ so that 
  \[
    |x_{j+1,q} - x_{j,q'}| \leq \epsilon_j.
  \]
  We then define
  \[
    w_{j+1,q} 
    \coloneqq
    u_{j+1,q} - F_{0,m^{j}}(\theta_{j,q'})
    = 
    (u_{j+1,q} - u_{j,q'}) + (F_{0,m^{j}}(\theta_{j,q'}) - u_{j,q'}),
  \]
  which functions as the homing target for $q$-th `particle' in the $(j+1)$-st step.
  We define the angular preimage $\theta_{j+1,q}$ of the same particle by
  \[
    \begin{aligned}
    &\operatorname{argmin}
    \Bigl\{|F_{m^j,m^{j+1}}(\theta) - w_{j+1,q}| + \Delta_{m^j,m^{j+1}}(\theta,4) : \theta \in \mathrm{Arc}
    \Bigr\}, \\
    &\mathrm{Arc} = 
    \{\theta : |e(\theta) - e(\theta_{j,q'})| \leq 2^{-Lm^j+1}\}\Bigr\}. 
    \end{aligned}
  \]
  We define the corresponding event that this angle fails to be good enough by
  \[
    E^{j+1}_q \coloneqq
    \Bigl\{|F_{m^j,m^{j+1}}(\theta_{j+1,q}) - w_{j+1,q}| + \Delta_{m^j,m^{j+1}}(\theta_{j+1,q},4) \geq 2( \frks_{m^{j+1}})^{1-\gamma}\Bigr\},
  \]
  and we also let $E^{j+1}$ be the union over all $q$ of these events.  We define $E^{j_0}$ by (recalling all $\theta_{j_0,q}$ are the same for $1\leq q \leq 2^{2^{j_0}}$)
  \[
    E^{j_0}
    \coloneqq
    \Bigl\{|F_{0,m^{j_0}}(\theta_{j_0,q}) - 0| + \Delta_{0,m^{j_0}}(\theta_{j_0,q},4) \geq 2( \frks_{m^{j_0}})^{1-\gamma}\Bigr\}.
  \]
  Finally, we define the good events for $j_0 \leq j \leq j_1$
  \[
    G^j = \Big\{
      \sup_{(\theta,r) \in \mathcal{W}_{m^j}(\theta_{j,q},1)}
    |F_{0,m^{j}}(\theta,r) - u^j_q| \leq 12( \frks_{m^{j}})^{1-\gamma}
      : \text{for all}~1 \leq q \leq 2^{2^j}
    \Big\}.
  \]

  Now we claim that $G^j \setminus (\bigcup_{\ell=j_0}^{j+1} E^{\ell} ) \subset G^{j+1}.$
  So we fix a $q$ with $1 \leq q \leq 2^{2^{j+1}}$.
  By iterating the application of the $q'(\cdot, \cdot)$ map, we define a sequence $q^{(\ell)}$ for $j_0 \leq \ell \leq j+1$ given by $q^{(\ell)} = q'( q^{(\ell+1)}, \ell)$ and with initial condition $q^{(j+1)} = q.$
  Then by construction for all $j_0 \leq \ell \leq j+1$
  \[
  |e(\theta_{\ell+1,q^{(\ell+1)}})
  -
  e(\theta_{\ell,q^{(\ell)}})| \leq 2^{-Lm^\ell+1}
  \]
  and \emph{off} the event $E^\ell$,
  \[
    |\Delta_{m^\ell,m^{\ell+1}}(\theta_{\ell+1,q^{(\ell+1)}},4)| \leq 2(\frks_{m^{\ell+1}})^{1-\gamma}.
  \]

  Now, for any $j_0 \leq \ell \leq j$ and any $(\theta,r) \in \mathcal{W}_{m^{j+1}}(\theta_{j+1,q},1)$
  \[
    \max\{|1-r|+|e(\theta) - e(\theta_{\ell+1,q^{(\ell+1)}})|,
    |e(\theta_{j,q^{(j)}}) - e(\theta_{\ell+1,q^{(\ell+1)}})|\}
    \leq 4 \cdot 2^{-Lm^{\ell+1}}.
  \]
  So, we have $(\theta, r)$ and $(\theta_{j,q^{(j)}},1)$ are in $\mathcal{W}_{m^{\ell+1}}(\theta_{\ell+1,q^{(\ell+1)}},4) \cap \mathcal{W}_{m^{j}}(\theta_{j,q^{(j)}},4)$,
  and thus off the event $E^{\ell+1}$ 
  \[
    \begin{aligned}
    &|F_{m^\ell,m^{\ell+1}}(\theta,r) - F_{m^\ell,m^{\ell+1}}(\theta_{j,q^{(j)}})|\\
    &\leq 
    \Delta_{m^\ell,m^{\ell+1}}(\theta_{\ell+1,q^{(\ell+1)}},4) 2^{-L(m^j-m^{\ell+1})/2+1} \\
    &\leq 4(\frks_{m^{\ell+1}})^{1-\gamma}2^{-L(m^j-m^{\ell+1})/2}.
    \end{aligned}
  \]
  Hence off the event $\bigcup_{\ell=j_0}^{j+1} E^{\ell}$
  and by \eqref{eq:domination}
  \[
    |F_{0,m^{j}}(\theta,r) - F_{0,m^{j}}(\theta_{j,q'})|
    \leq
    4\sum_{\ell = j_0}^{j} (\frks_{m^{\ell}})^{1-\gamma}2^{-L(m^j-m^{\ell})/2}
    \leq 8(\frks_{m^{j+1}})^{1-\gamma},
  \]
  and so for any $(\theta,r) \in \mathcal{W}_{m^{j+1}}(\theta_{j+1,q},1)$,
  \[
    \begin{aligned}
    &|F_{0,m^{j+1}}(\theta,r) - u_{j+1,q}| \\
    &\leq
    |F_{m^j,m^{j+1}}(\theta,r) - F_{m^j,m^{j+1}}(\theta_{j+1,q})|+
    |F_{m^j,m^{j+1}}(\theta_{j+1,q}) - w_{j+1,q}|\\
    & + |F_{0,m^{j}}(\theta,r) - F_{0,m^{j}}(\theta_{j,q'}) | \\
    &\leq \Delta_{m^j,m^{j+1}}(\theta_{j+1,q},4)\sqrt{1+1}
    +|F_{m^j,m^{j+1}}(\theta_{j+1,q}) - w_{j+1,q}|+8( \frks_{m^{j+1}})^{1-\gamma} \\
    &\leq (2\sqrt{2}+8)( \frks_{m^{j+1}})^{1-\gamma}.
    \end{aligned}
  \]

  This concludes the claim that $G^j \setminus (\bigcup_{\ell=j_0}^{j+1} E^{\ell} ) \subset G^{j+1}.$
  It follows that (using that $G^{j_0}$ contains $(E^{j_0})^c$)
  \begin{equation}
    \Pr\Big( G^{j_1} \setminus \big(\bigcup_{\ell=j_0}^{j_1} E^{\ell} \big) \Big)
    \geq 1- \Pr(E^{j_0}) - \sum_{j=j_0}^{j_1-1} \sum_{q=1}^{2^{2^{j+1}}} 
    \Pr( E_q^{j+1} \cap G^{j}).
    \label{eq:Gj}
  \end{equation}

  \paragraph{Bounding the probabilities.}

  For the initial event $E^{j_0}$ as all of the $\theta_{j_0,q}$ are equal and $u_{j_0,q}=0$ for all $q$, this event is
  \[
    E^{j_0} = \{ \exists~ \theta \in [0,2\pi] ~:~ |F_{0,m^{j_0}}(\theta)| +\Delta_{0,m^{j_0}}(\theta,4)  \geq  2( \frks_{m^{j_0}})^{1-\gamma}
  \}.
  \]
  By running the homing process to $u=0$ with initial angle $0$, Proposition \ref{prop:homing}, over blocks $0$ to $m^{j_0}$,
  \[
    \Pr(E^{j_0})
    \leq \exp(-c \min\{ L_{0,m^{j_0}}^2/V_{0,m^{j_0}}, (m^{j_0})^{\gamma/4}\}) \leq \epsilon/3. 
  \]

  As for $j \geq j_0$, to control the probability of $E_q^{j+1} \cap G^{j}$, we run the homing process to $u = w_{j+1,q}$ (which is adapted to $\filt_{m^j}$), started from $\theta_{j,q'}$ over the blocks $m^j$ to $m^{j+1}$.
  To apply Proposition \ref{prop:homing}, we require $|u| \leq \tfrac 12 L_{m^j,m^{j+1}}.$
  On the event $G^j$, we have
  \[
    \begin{aligned}
    | w_{j+1,q}|
    &=
    |(u_{j+1,q} - u_{j,q'}) + (F_{0,m^{j}}(\theta_{j,q'}) - u_{j,q'})| \\
    &\leq  |u_{j+1,q} - u_{j,q'}| + 12 (\frks_{m^j})^{1-\gamma} \\
    &\leq  |b_{j+1} x_{j+1,q} + z_{j+1} - (b_j x_{j,q'} + z_j)| + 12 (\frks_{m^j})^{1-\gamma} \\
    &\leq (|z_{j+1}-z_j| + | b_{j+1}-b_j|) + \epsilon_j + 12 (\frks_{m^j})^{1-\gamma} \\
    &\leq \tfrac16L_{m^j,m^{j+1}} + \tfrac16L_{m^j,m^{j+1}} + \tfrac16L_{m^j,m^{j+1}}.
  \end{aligned}
  \]
  Thus Proposition \ref{prop:homing} applies and
  \[
    \begin{aligned}
    &\sum_{j=j_0}^{j_1-1} \sum_{q=1}^{2^{2^{j+1}}} 
    \Pr( E_q^{j+1} \cap G^{j})\\
    &\leq
    \sum_{j=j_0}^\infty
    2^{2^{j+1}}
    \exp(-c \min\{ L_{m^j,m^{j+1}}^2/V_{m^{j},m^{j+1}}, (m^{j})^{\gamma/4}\})
    \leq \epsilon/3.
    \end{aligned}
  \]
\end{proof}

\begin{proof}[Proof of Theorem~\ref{thm:coverage} in the case $\beta=1$]
  It suffices to show that for any $u \in \C$ and any $\delta > 0$
  \[
    \Pr( F(\D) \supset \D(u,1)) \geq 1 - \delta.
  \]
  We begin by applying Proposition \ref{prop:homingnet} with $\epsilon = \delta/2$, and $u, j$ to be determined, which gives the existence of a net $\{u_q\}$ of $\D(u,1)$ and special angles $\{\Theta_q\}$.  From this net, we pick a $(m^j)^{-4}$--subnet $\{u_q\}$ of cardinality $C(m^j)^8$ (which is possible for all $j$ sufficiently large) which satisfies
  \begin{equation}\label{eq:preEh}
    \sup_{ (\theta,r) \in \mathcal{W}_{m^j}(\Theta_q,1)} |F_{0,m^j}(\theta, r) - u_q |
    \leq 12 (\frks_{m^j})^{1-\gamma}.
  \end{equation}

  Now we apply Proposition \ref{prop:cover} for each $q$ one at a time, with parameters ($k_h = m^j$, $u=u_q$, $\epsilon = \tfrac 12$ and $R=10$) and working on the good event of Proposition \ref{prop:homingnet}. These apply since $12( \frks_{m^{j+1}})^{1-\gamma} \leq (m^j)^{-3/4}$ for all $j$ sufficiently large (as $\gamma \in (0,\tfrac14)$).  We note that the event \eqref{eq:Eh} holds because of \eqref{eq:preEh}.  The conclusion is that provided $m^j$ is sufficiently large,
  \[
    \Pr(  F(\D) \supset \D(u_q,(m^j)^{-7/2}) \mid \filt_{m^j}) \geq 1-(m^j)^{-10}.
  \]
  By a union bound, we conclude for all $j$ sufficiently large
  \[
    \Pr(  F(\D) \supset \D(u,1) \mid \filt_{m^j}) \geq 1-(m^j)^{-2} \geq 1-\delta/2,
  \]
  which completes the proof.

\end{proof}

\appendix

\section{Dense image of random walks}\label{app:rw}

\begin{lemma}\label{lem:bmcoupling}
  Suppose that the sequence $\left\{ (r_k,\omega_k) \right\}$ are independent random variables with $r_k > 0$ a.s.,\ $\omega_k$ uniform on $\T,$ and $r_k$ independent of $\omega_k$ for each $k.$  Suppose further that $k^{\epsilon}r_k \Asto 0$ for some $\epsilon > 0$ and $\sum r_k^2 = \infty$ almost surely.  
  Then the set $\{S_k : k\}$ is almost surely dense in $\C$ where
  \[
    S_k \coloneqq \sum_{j=1}^k r_j \omega_j.
  \]
\end{lemma}
\begin{proof}
  
  We give a proof by comparison with complex Brownian motion $B_t$ (normalized such that $\Exp[|B_t|^2]=t$).
  We enlarge the probability space to include a standard complex
  $(B_t : t \geq 0)$ adapted to a filtration $(\filt_t : t\geq 0),$ and we can define stopping times $\left\{ \tau_k : k \geq 0 \right\}$ with $\tau_0 = 0$ so that
  \begin{enumerate}[(i)]
    \item
      For all $k \in \N,$
      \[
	|B_{\tau_{k}} - B_{\tau_{k-1}}| = r_k.
      \]
    \item $\tau_k \uparrow \infty$ almost surely.
    \item 
      $\max_{t \in [\tau_{k-1},\tau_k]} |B_t - B_{\tau_{k-1}}| \Asto[k] 0.$
    \item The sequence $\left\{ B_{\tau_k} = S_k : k \in \N \right\}$ is dense in $\C$.
  \end{enumerate}
  Define the stopping times by
  \[
    \tau_{k} = \inf \left\{ t > \tau_{k-1} : |B_t - B_{\tau_{k-1}}| = r_{k} \right\}.
  \]
  By rotational invariance of the Brownian motion, we further have that $r_k^{-1} (B_{\tau_{k}} - B_{\tau_{k-1}})$ are independent and uniformly distributed on $\T$.  

  As $|B_t - B_{\tau_k}|^2 - t$ is a martingale, optional stopping gives that $\Exp[\tau_{k} - \tau_{k-1} \vert r_k] = r_k^2.$  It follows that conditionally on $\Gfilt = \sigma( (r_k) : k \in \N),$
  \[
    \sum_{k=1}^\infty \Exp[ \tau_{k} - \tau_{k-1} \vert \Gfilt] = \infty.
  \]
  The variables $r_k^{-1}(\tau_{k} - \tau_{k-1})$ are i.i.d.\ and have exponential tails.  Hence as $k^\epsilon r_k \to 0,$  it follows that
  \[
    \sum_{k=1}^\infty \Pr(\tau_{k} - \tau_{k-1} \geq 1  \vert \Gfilt) < \infty
    \quad\text{and}
    \quad
    \sum_{k=1}^\infty \Exp[ (\tau_{k} - \tau_{k-1}) \one_{(\tau_{k} - \tau_{k-1}) \leq 1} \vert \Gfilt] = \infty
  \]
  almost surely. It follows that (c.f.\ \cite[Exercise 2.5.5]{Durrett} or \cite[Chap. 3, Thm.~6]{Kahane})
  \[
    \sum_{k=1}^\infty (\tau_{k} - \tau_{k-1}) = \infty \quad \As
  \]
  for $\Pr$-almost-every realization of $\{r_k\}$. 

  Hence, from the neighborhood recurrence of $(B_t : t \geq 0),$ and since  $\max_{t \in [\tau_{k-1},\tau_k]} |B_t - B_{\tau_{k-1}}| \Asto[k] 0,$ it follows that the set $\left\{ B_{\tau_k} : k \in \N \right\}$ is dense. By equality is distribution the set $(S_k)_k$ is also dense in $\C$.
\end{proof}

\printbibliography[heading=bibliography]

\end{document}